\documentclass[preprint]{elsarticle}
\usepackage{amsmath}
\usepackage{amssymb}
\usepackage{amsthm}
\usepackage{graphicx}
\usepackage{epic,eepic,epsfig}
\usepackage{color}
\usepackage{subfigure}	
\usepackage{placeins}	
\usepackage{multirow}
\usepackage{epstopdf}
\usepackage{varwidth}
\usepackage[table]{xcolor}

\newtheorem{theorem}{Theorem}
\newtheorem{lemma}{Lemma}

\newcommand{\revHH}[1]{\textcolor{black} {#1}}

\definecolor{tabclr}{cmyk}{0,0,1,0}

\begin{document}

\title{BDDC and FETI-DP algorithms with adaptive coarse spaces
for three-dimensional elliptic problems with oscillatory and high contrast coefficients}

\author[KHU]{Hyea Hyun Kim\corref{cor}\fnref{fn1}}
\ead{hhkim@khu.ac.kr}

\author[CUHK]{Eric Chung\fnref{fn2}}
\ead{tschung@math.cuhk.edu.hk}


\author[CUHK,XTU]{Junxian Wang\fnref{fn3}}
\ead{wangjunxian@xtu.edu.cn}

\cortext[cor]{Corresponding author}
\fntext[fn1]{The research of Hyea Hyun Kim is
supported by the National Research Foundation of Korea(NRF) grants
funded by NRF-2014R1A1A3052427 and by NRF20151009350.}
\fntext[fn2]{The research of Eric Chung is supported by Hong Kong RGC General Research Fund (Project number 400813)
and CUHK Faculty of Science Research Incentive Fund 2015-16.}
\fntext[fn3]{The research of Junxian Wang is supported by National Natural Science Foundation of China (Project number 11201398, 11301448).}
\address[KHU]{Department of Applied Mathematics and Institute of Natural Sciences, Kyung Hee University,
Korea}
\address[CUHK]{Department of Mathematics, The Chinese University of Hong Kong, Hong Kong SAR}

\address[XTU]{School of Mathematics and Computational Science, Xiangtan University, Xiangtan, Hunan 411105, China}

\begin{abstract}
BDDC and FETI-DP algorithms are developed for three-dimensional elliptic problems
with adaptively enriched coarse components.
It is known that these enriched components are necessary in the development of robust preconditioners.
To form the adaptive coarse components, carefully designed generalized eigenvalue problems
are introduced for each faces and edges, and the coarse components
are formed by using eigenvectors with their corresponding eigenvalues larger than a given tolerance
$\lambda_{TOL}$. Upper bounds for condition numbers of the preconditioned systems
are shown to be $C \lambda_{TOL}$, with the constant $C$
depending only on the maximum number of edges and faces per subdomain, and the maximum number of subdomains
sharing an edge.
Numerical results are presented to test the robustness of the proposed approach.
\end{abstract}


\maketitle

\section{Introduction}
In this paper, we construct and analyze a class of \revHH{domain decomposition} preconditioners
for \revHH{fast solutions of} the finite element approximation of elliptic problems,
\begin{equation}\label{model:pb}
\begin{split}
-\nabla \cdot ( \rho \nabla u) &= f, \quad \text{ in } \Omega, \\
u &= 0, \quad \text{ on } \partial\Omega,
\end{split}
\end{equation}
where $\Omega$ is a bounded domain in $\mathbb{R}^d$, $d=2,3$, and $\rho(x)$ is uniformly positive and is
highly heterogeneous with very high contrast.
\revHH{In the domain decomposition methods, the domain $\Omega$ is partitioned into subdomains
and the original problem is solved iteratively by solving independent subdomain problems
and a global coarse problem at each iteration.
Such a process provides a preconditioner to the original problem.
The role of the coarse problem is important in obtaining the convergence of an iterative method
robust to the number of subdomains. We refer to \cite{TW-Book} for a general introduction to
domain decomposition preconditioners.}

\revHH{It is well known that the convergence of domain decomposition preconditioners
can be affected by the heterogeneity of $\rho(x)$ across the subdomain interfaces.
For some special cases, the standard coarse problems formed by unknowns at subdomain vertices,
edge averages, or face averages can give robust preconditioners with appropriate scaling factors~\cite{PS-FETI-multiscale1,PS-FETI-multiscale2}.
In general for a bad arrangement of coefficients, the standard coarse problems are not enough
and they can be enriched by choosing adaptive primal constraints.
One possible approach is to select the adaptive constraints from
local generalized eigenvalue problems.
The local generalized eigenvalue problems are a good indicator for the bad behavior of the standard
coarse problem and they thus can be used to select the primal constraints to enhance the
convergence of the iteration. We refer to \cite{KRR-2015} and references therein for general reviews
on the approaches in this direction.}


\revHH{In this paper, we will develop domain decomposition preconditioners with adaptively
enriched coarse problems.}
To be more specific, we will consider two types of domain decomposition techniques, namely
the BDDC (Balancing Domain Decomposition by Constraints) algorithm
and the FETI-DP (Dual-Primal Finite Element
Tearing and Interconnecting) algorithm; see \cite{BDDC-D-03,FETI-DP-Farhat,TW-Book,LW-FETIDP-BDDC}.
For these algorithms, we will develop their coarse components with an adaptively chosen set of primal unknowns,
which are robust to the coefficient variations.
\revHH{So far complete theories and numerical validations have been successfully developed for two-dimensional problems~\cite{DC-dd21-talk,Pechstein:slide:2013,Klawonn-DD21,DD22-Klawonn,Klawonn:PAMM:2014,Klawonn-TR,KimChung:2014,KRR-2015}. On the other hand, three-dimensional extension of the existing methods
with complete theory is still under development. The main contribution of the current paper is the extension of the existing methods in \cite{Pechstein:slide:2013,Klawonn:PAMM:2014,KimChung:2014,KRR-2015} to three-dimensional problems with a complete theory.} We note that in \cite{Mandel-ABDDC} an adaptive BDDC algorithm was developed and numerically studied for three-dimensional problems. More recently at the 23rd international conference on domain decomposition methods
considerable progresses on the three-dimensional problems were presented in talks by Clark Dohrmann, Axel Klawonn, and Olof Widlund.
In the talk of Olof Widlund, similar approach to ours was presented to deal with
equivalence classes sharing more than two subdomains.
His work and ours have been independently developed, being unaware of each others.
We refer to \cite{S-R-ref-2013,Dolean3} for FETI/BDD type methods and to \cite{Galvis1,Galvis2,Dolean1,Dolean2} for other variants of domain decomposition methods
with adaptively enriched coarse spaces.

Specifically, two types of preconditioning techniques are considered and analyzed in this paper.
The first method is the BDDC algorithm with a change of basis formulation
and with adaptively chosen primal unknowns by solving generalized eigenvalue problems.
The second method is
the FETI-DP algorithm with a projector preconditioner
and with adaptively chosen primal unknowns by solving the same generalized eigenvalue problems.
\revHH{The generalized eigenvalue problems are formed for each face and each edge.
The face is an equivalence class shared by two subdomains and thus the generalized eigenvalue problem
is identical to that considered for two-dimensional problems in \cite{Pechstein:slide:2013,KimChung:2014,Klawonn:PAMM:2014,KRR-2015}.
On the other hand, the edge is an equivalence class shared by more than two subdomains and thus
a different idea is required to form an appropriate generalized eigenvalue problem.}
For both methods, we will show that the condition numbers can be controlled by
$C \lambda_{TOL}$, where $\lambda_{TOL}$ is a given
tolerance used to choose the problematic eigenvectors in the generalized eigenvalue problems
and $C$ is a constant
depending only on the number of edges and faces per subdomain, and the number of
subdomains sharing an edge.
\revHH{Our generalized eigenvalue problems are based on that proposed in \cite{Klawonn:PAMM:2014}.
In that approach, the scaling matrices of FETI-DP and BDDC preconditioners come in the generalized eigenvalue problems and they can help to reduce the bad eigenvectors and result in a smaller set of
adaptive constraints. We refer to \cite{KRR-2015} for various numerical examples
and references therein for scaling matrices.}
In our numerical experiments, we also observe that the use of deluxe scalings reduces the number of
problematic eigenvectors.

\revHH{In our work, one important observation is that
in the three-dimensional experiments using a larger tolerance value
we can choose a more effective set of primal constraints on edges.
This is due to the fact that the right hand side of the generalized eigenvalue problem in \eqref{eq:Geig:edge} for the edge
underestimates the energy of a subdomain problem as more than two subdomains are
involved in the right hand side, see also \eqref{eq:matrixineq}.
We also observe that the deluxe scaling is
less sensitive to the tolerance value and keeps a good bound of condition numbers
even for a larger tolerance value.
Another important issue is computational efficiency.
Though the condition numbers can be controlled by the user-defined tolerance,
the cost for forming generalized eigenvalue problems is quite considerable especially for three-dimensional problems.
Thus a similar idea to \cite{KRR-2015}, using e-version (economic-version), can be applied to enhance the efficiency of the proposed method for the three-dimensional problems.}


This paper is organized as follows.
In Section~\ref{intro:BDDC:FETI}, a brief introduction to BDDC and FETI-DP  methods with adaptively enriched coarse problems is presented for two-dimensional elliptic problems.
In Section~\ref{ext:3D}, three-dimensional extension of these methods is carried out.
Analysis of condition numbers is provided in Section~\ref{analysis} and various numerical experiments are
presented in Section~\ref{sec:num}.

\section{\label{intro:BDDC:FETI} Adaptively enriched coarse problems in BDDC and FETI-DP}
In this section, we will give an overview of BDDC and FETI-DP methods
with the use of adaptively enriched coarse spaces.
\revHH{We first introduce a discrete form of the model problem~\eqref{model:pb}.
Let $V_h$ be the space of conforming linear finite element functions
with respect to a given mesh on $\Omega$ with mesh size $h>0$ and with the zero value on $\partial \Omega$.
We will then find the approximate solution $u\in V_h$ such that
\begin{equation}
a(u,v) = (f,v), \quad\quad \forall v\in V_h,
\label{eq:cg}
\end{equation}
where
\begin{equation}
a(u,v) = \int_{\Omega} \rho(x) \nabla u \cdot \nabla v \, dx, \quad
(f,v) = \int_{\Omega} f \, v \, dx.
\end{equation}
We assume that the domain $\Omega$ is partitioned into a set of $N$ non-overlapping subdomains $\{ \Omega_i \}$, $i=1,2,\cdots, N$,
so that $\Omega = \cup_{i=1}^N \Omega_i$.
We note that the subdomain boundaries do not cut triangles equipped for $V_h$.
We allow the coefficient $\rho(x)$ to have high contrast jumps and oscillations across subdomains
and on subdomain interfaces.}
Let $a_i(u,v)$ be the bilinear form of the model elliptic problem (\ref{eq:cg}) restricted to each subdomain $\Omega_i$ defined as
$$a_i(u,v)=\int_{\Omega_i} \rho(x) \nabla u \cdot \nabla v \, dx, \quad \forall u, v \in X_i,$$
where $X_i$ is the restriction of $V_h$ to $\Omega_i$.

\revHH{In the BDDC and FETI-DP algorithms, the original problem~\eqref{eq:cg} is solved by
an iterative method combined with a preconditioner. In the BDDC algorithm, the original problem
is reduced to a subdomain interface problem. The interface problem can be obtained by
solving a Dirichlet problem in each subdomain. After choosing dual and primal unknowns on the subdomain
interface unknowns, the interface problem is then solved by utilizing local problems and one global
coarse problem corresponding to the chosen sets of dual and primal unknowns, respectively.
At each iteration, certain scaling factors are multiplied to the residual vectors to balance
the errors across the subdomain interface regarding to the energy of each subdomain problem.
The coarse problem corrects the global part of the error in each iteration and thus the choice of primal unknowns is important in obtaining a good performance as increasing the number of subdomains.
The basis for primal unknowns is obtained by the minimum energy extension for a given constraint
at the location of primal unknowns and such a basis provides a robust coarse problem with a good
energy estimate. The FETI-DP algorithm is similar to the BDDC algorithm except that it is a dual counterpart
of the BDDC algorithm. In that algorithm, the interface problem is first decoupled at the dual unknowns
and then coupled at the primal unknowns. The continuity at the decoupled dual unknowns is enforced weakly
by Lagrange multipliers. The unknowns except the Lagrange multipliers are eliminated by solving
local problems and one global coarse problem. The resulting system on the Lagrange multipliers
is then solved by an iterative method with a preconditioner. We refer to \cite{FETI-DP-Farhat,BDDC-D-03,BDDC-Mandel-Dohrmann-Tezaur,LW-FETIDP-BDDC,TW-Book} for general
introductions to FETI-DP and BDDC algorithms.}


\subsection{\label{BDDC:FETI:theory}Notations and preliminary results}
To facilitate our discussion, we first introduce some notations.
Let $S^{(i)}$ be the Schur complement matrix obtained from the local stiffness matrix $A^{(i)}$
after eliminating unknowns interior to $\Omega_i$, where $A^{(i)}$ is defined by
$a_i(u,v)=u^T A^{(i)} v$, for all $u,v\in X_i$.
In the following we will use the same symbol to represent a finite element function and its corresponding coefficient vector
in order to simplify the notations.

Recall that $X_i$ is the restriction of the finite element space $V_h$ to each subdomain $\Omega_i$.
Let $W_i$ be the restriction of $X_i$ to $\partial \Omega_i$.
We then introduce the product spaces
$$X=\prod_{i=1}^N X_i,\quad W=\prod_{i=1}^N W_i,$$
where we remark that the functions in $X$ and $W$ are totally decoupled across the subdomain interfaces.
In addition, we introduce partially coupled subspaces $\widetilde{X}$, $\widetilde{W}$,
and fully coupled subspaces $\widehat{X}$, $\widehat{W}$,
where some primal unknowns are strongly coupled for functions in $\widetilde{X}$ or $\widetilde{W}$,
while the functions in $\widehat{X}$, $\widehat{W}$ are fully coupled across the subdomain interfaces.

Next, we present basic description of the BDDC algorithm; see~\cite{BDDC-D-03,BDDC-Mandel-Dohrmann-Tezaur,LW-FETIDP-BDDC,TW-Book}.
For simplicity, the two-dimensional case will be considered.
After eliminating unknowns interior to each subdomain, the Schur complement matrices $S^{(i)}$
are obtained from $A^{(i)}$ and they form the algebraic problem considered in the BDDC algorithm,
which is to find $\widehat{w} \in \widehat{W}$ such that
\begin{equation}\label{primal-pb}\sum_{i=1}^N R_i^T S^{(i)} R_i \widehat{w}= \sum_{i=1}^N R_i^T g_i,
\end{equation}
where $R_i: \widehat{W} \rightarrow W_i$ is the restriction operator into $\partial\Omega_i$,
and $g_i \in W_i$ depends on the source term $f$.

The BDDC preconditioner is built based on the partially coupled space $\widetilde{W}$.
Let $\widetilde{R}_i: \widetilde{W} \rightarrow W_i$ be the restriction into $\partial \Omega_i$
and let $\widetilde{S}$ be the partially coupled matrix defined by
$$\widetilde{S}=\sum_{i=1}^N \widetilde{R}_i^T S^{(i)} \widetilde{R}_i.$$
For the space $\widetilde{W}$, we can express it as the product of the two spaces
$$\widetilde{W}=W_{\Delta} \times \widehat{W}_{\Pi},$$
where $\widehat{W}_{\Pi}$ consists of vectors of the primal unknowns and
$W_{\Delta}$ consists of vectors of dual unknowns, which are strongly coupled
at the primal unknowns and decoupled at the remaining interface unknowns, respectively.
We define $\widetilde{R}: \widehat{W} \rightarrow \widetilde{W}$ such that
$$\widetilde{R}=\begin{pmatrix}
R_{\Delta} \\
R_{\Pi}
\end{pmatrix},$$
where $R_{\Delta}$ is the mapping from $\widehat{W}$ to $W_{\Delta}$ and $R_{\Pi}$ is the restriction
from $\widehat{W}$ to $\widehat{W}_{\Pi}$.
We note that $R_{\Delta}$ is obtained as
$$R_{\Delta}=\begin{pmatrix} R_{\Delta}^{(1)} \\ R_{\Delta}^{(2)}\\ \vdots \\ R_{\Delta}^{(N)}
\end{pmatrix},$$
where $R_{\Delta}^{(i)}$ is the restriction from $\widehat{W}$ to $W_{\Delta}^{(i)}$
and $W_{\Delta}^{(i)}$ is the space of dual unknowns of $\Omega_i$.

The BDDC preconditioner is then given by
\begin{equation}\label{precond:BDDC}
M^{-1}_{BDDC}=\widetilde{R}^T \widetilde{D} \widetilde{S}^{-1} \widetilde{D}^T \widetilde{R},
\end{equation}
where $\widetilde{D}$ is a scaling matrix of the form
\revHH{$$\widetilde{D}=\sum_{i=1}^N \widetilde{R}_i^T D_i \widetilde{R}_i.$$
Here the matrices  $D_i$ are defined for unknowns in $W_i$ and they
are introduced to resolve heterogeneity in $\rho(x)$ across the subdomain interface.
In more detail, $D_i$ consists of blocks $D_F^{(i)}$ and $D_V^{(i)}$, where $F$ denotes an equivalence class
shared by two subdomains, i.e., $\Omega_i$ and its neighboring subdomain $\Omega_j$, and
$V$ denotes the end points of $F$, respectively.}
We call such equivalence classes $F$ edge in two dimensions while they are called face in
three dimensions. In three dimensions, equivalence classes shared by more than two subdomains
are called edge. In both two and three dimensions, vertices are equivalence classes which are end points
of edges. We refer to \cite{klawonn2006dual} for these definitions.
\revHH{In our BDDC algorithm, unknowns at subdomain vertices are included to the set of primal unknowns
and adaptively selected primal constraints are later included to the set after a change of basis formulation.}
\revHH{For a given edge $F$ in two dimensions, the matrices $D_F^{(l)}$ and $D_V^{(l)}$
satisfy a partition of unity property, i.e., $D_F^{(i)}+D_F^{(j)}=I$ and $\sum_{l \in n(V) } D_V^{(l)}=1$,
where $n(V)$ denotes the set of subdomain indices sharing the vertex $V$. The matrices $D_F^{(l)}$ and $D_V^{(l)}$
are called scaling matrices.}
\revHH{As mentioned earlier, the scaling matrices help to balance the residual error
at each iteration with respect to the energy of subdomain problems sharing the interface.
For the case when $\rho(x)$ is identical across the interface $F$, $D_F^{(i)}$ and $D_F^{(j)}$ are chosen
simply as multiplicity scalings, i.e., $1/2$, but for a general case when $\rho(x)$ has discontinuities,
different choice of scalings, such as $\rho$-scalings or deluxe scalings, can be more effective.
The scaling matrices $D_V^{(l)}$ can be chosen using similar ideas.
We refer to \cite{KRR-2015} and references therein for scaling matrices.}

We note that by using the definitions of $\widetilde{R}$ and $\widetilde{S}$, the matrix in the left hand side of \eqref{primal-pb}
can be written as
$$\sum_{i=1}^N R_i^T S^{(i)} R_i = \widetilde{R}^T \widetilde{S} \widetilde{R}.$$
In the BDDC algorithm, the system in \eqref{primal-pb} is solved by an iterative method
with the preconditioner~\eqref{precond:BDDC}. Thus its performance is analyzed by estimating
the condition number of
\begin{equation}\label{cond:BDDC}M^{-1}_{BDDC} \widetilde{R}^T \widetilde{S} \widetilde{R}=\widetilde{R}^T \widetilde{D} \widetilde{S}^{-1} \widetilde{D}^T \widetilde{R}\widetilde{R}^T \widetilde{S} \widetilde{R}.
\end{equation}

We now present basic ideas of the FETI-DP algorithm; see \cite{FETI-DP-Farhat,LW-FETIDP-BDDC,TW-Book}.
\revHH{We confine our presentation to the case when primal unknowns at subdomain vertices
are only considered. In the FETI-DP algorithm, the adaptive set of primal constraints will be enforced by
using a projection.}
The model elliptic problem (\ref{eq:cg})
is solved in the partially coupled space $\widetilde{W}$ with the continuity
on the decoupled unknowns in $W_{\Delta}$ enforced weakly by introducing Lagrange multipliers $\lambda$.
We call the remaining decoupled unknowns dual unknowns.
Specifically, we consider the following system
\begin{align*}
\widetilde{S} \widetilde{w} + B^T \lambda&= \widetilde{g}, \\
B \widetilde{w} &=0,
\end{align*}
where $\widetilde{w} \in \widetilde{W}$, \revHH{$\widetilde{g} \in \widetilde{W}$ is the partially assembled vector
of $g_i$ at the subdomain vertices}, and $B=\begin{pmatrix} B_{\Delta} & 0 \end{pmatrix}$
with the blocks $B_{\Delta}$ and $0$ corresponding to dual unknowns and primal unknowns, respectively.
In addition,
$B_{\Delta} {w}_{\Delta} |_{F}=w_{F}^{(i)}-w_{F}^{(j)}$, for the common part $F$ of $\partial \Omega_i$ and $\partial \Omega_j$
with $w_{F}^{(i)}$ being the part of unknowns of $w_{F}^{(i)} \in W_{\Delta}^{(i)}$
interior to the edge $F$ excluding the two end points.
The matrix $B_{\Delta}$ consists of blocks $B_{\Delta}^{(i)}$,
$$B_{\Delta}=\begin{pmatrix} B_{\Delta}^{(1)} & \cdots & B_{\Delta}^{(N)} \end{pmatrix},$$
where each block corresponds to unknowns in each subdomain and the blocks consist of
entries 0, 1, or -1.
We introduce $\text{Range}(B)$ as the space of Lagrange multipliers $\lambda$. We note that when
the constraints $B \widetilde{w}=0$ are not redundant, $\text{Range}(B)$ is identical to $\mathbb{R}^{M_N}$, where $M_N$ is
the total number of constraints in $B \widetilde{w}=0$. On the other hand, when redundant constraints
are employed, $\text{Range}(B)$ is a proper subspace of $\mathbb{R}^{M_N}$, which will be the case in the three dimensions when fully redundant continuity constraints are enforced on edges.
For the two-dimensional case, equivalence classes shared by two subdomains are considered in $B \widetilde{w}=0$ after enforcing strong continuity on the unknowns at subdomain vertices and thus $\text{Range}(B)=\mathbb{R}^{M_N}$.
After eliminating $\widetilde{w}$ from the above system, the following algebraic system is obtained
\begin{equation}\label{sys:Fdp}
B \widetilde{S}^{-1} B^T \lambda = d,
\end{equation}
\revHH{where $d=B \widetilde{S}^{-1} \widetilde{g}$} and it is solved by an iterative method with the following preconditioner
\begin{equation}\label{precond:Fdp}
M^{-1}_{FETI}=B_D \widetilde{S} B_D^T ,
\end{equation}
where $B_D$ is  defined by
\begin{equation}\label{BTDs}B_D=\begin{pmatrix} B_{D,\Delta} & 0 \end{pmatrix}
=\begin{pmatrix} B_{D,\Delta}^{(1)} & \cdots & B_{D,\Delta}^{(i)} & 0 \end{pmatrix}.
\end{equation}
In the above, $B_{D,\Delta}^{(i)}$ is a scaled matrix of $B_{\Delta}^{(i)}$
where rows corresponding to Lagrange multipliers to unknowns $w_{F}^{(i)}$ are multiplied with
a scaling matrix $(D_{F}^{(j)})^T$ when $\Omega_j$ is the neighboring subdomain sharing the interface $F$ of $\partial \Omega_i$ and the Lagrange multipliers connect $w_F^{(i)}$ to $w_F^{(j)}$.
When the scaling matrices $D_F^{(l)}$ are the same as those in the BDDC algorithm,
it is well-known that FETI-DP and BDDC algorithms with the same set of primal unknowns
share the same set of spectra except zero and one; see \cite{Brenner-Sung,LW-FETIDP-BDDC}.

In the FETI-DP algorithm, \revHH{a condition number bound} is
analyzed for the following matrix,
\begin{equation}\label{cond:FETI}
M^{-1}_{FETI} B \widetilde{S}^{-1} B^T=B_D \widetilde{S} B_D^T B \widetilde{S}^{-1} B^T.
\end{equation}
Notice that the above discussions on FETI-DP/BDDC algorithms apply to the two-dimensional case.
The three-dimensional extension will be presented in Section~\ref{3d-constraints-weights}.

To conclude this section, we will state the main inequality
which is essential in our analysis of condition number bounds.
Let
$$G_p=\widetilde{R}\widetilde{R}^T \widetilde{D} \widetilde{S}^{-1} \widetilde{D}^T
\widetilde{R} \widetilde{R}^T \widetilde{S}$$
and
$$G_d=B^T B_D \widetilde{S} B_D^T B \widetilde{S}^{-1}.$$
An estimate of the condition number for the BDDC algorithm can be done using the matrix $G_p$,
since $G_p$ and $M^{-1}_{BDDC} \widetilde{R}^T \widetilde{S} \widetilde{R}$ share
the same set of non-zero eigenvalues.
Similarly, an estimate of the condition number for the FETI-DP algorithm can be done using the matrix $G_d$,
since $G_d$ and $M^{-1}_{FETI} B \widetilde{S}^{-1} B^T$ share
the same set of non-zero eigenvalues.
Let
$$E_D=\widetilde{R} \widetilde{R}^T \widetilde{D},\quad P_D=B_D^T B.$$
Then $G_p$ and $G_d$ can be written as
$$G_p=E_D \widetilde{S}^{-1} E_D^T \widetilde{S},
\quad G_d=P_D^T \widetilde{S}  P_D \widetilde{S}^{-1}.$$
Note that $E_D$ and $P_D$ satisfy
$$E_D+P_D=I.$$
Hence, $G_p$ and $G_d$ share the same set of eigenvalues except zero and one.
In addition, for $G_d$, it is known that all the nonzero eigenvalues are bounded below by one.
In conclusion, in the analysis of condition numbers of both the BDDC and the FETI-DP algorithms,
we only need to estimate an upper bound of
$G_d$.
Thus, we need to prove the following inequality
\begin{equation}\label{FETI-P} \langle \widetilde{S} P_D  \widetilde{w}, P_D \widetilde{w} \rangle
\le C \langle \widetilde{S} \widetilde{w}, \widetilde{w} \rangle.\end{equation}
\revHH{We note that the approach to reduce the condition number estimate to the above estimate
of the $P_D$ was first used in \cite{klawonn2006dual}.}
For more details regarding the above theories, we refer to \cite{Brenner-Sung,LW-FETIDP-BDDC}.

\revHH{On the other hand, when adaptive primal unknowns are introduced for the unknowns in $F$
after the change of basis formulation
the identity $E_D + P_D=I$ in the above does not hold in general.
Thus the two algorithms after the change of basis formulation do not satisfy the properties in
\cite{Brenner-Sung,LW-FETIDP-BDDC}.
For the adaptive primal constraints, we consider the BDDC algorithm
after a change of basis formulation and the FETI-DP algorithm with a projector preconditioning.
In the FETI-DP algorithm with a projector preconditioning, the primal constraints at subdomain vertices are enforced strongly while the adaptive primal constraints are enforced by using a projection.
For each method, we will provide estimate of condition numbers.
For the BDDC algorithm, we will need to estimate the following
inequality
$$ \langle \widetilde{S} (I-E_D ) \widetilde{w}, (I-E_D) \widetilde{w} \rangle
\le C \langle \widetilde{S} \widetilde{w}, \widetilde{w} \rangle,$$
while for the FETI-DP algorithm we will need to show the estimate in \eqref{FETI-P} for $\widetilde{w}$ continuous at the subdomain vertices
and satisfying the adaptive constraints, which are enforced by the projection.}


\subsection{Generalized eigenvalue problems}
In this section,
we review previous studies on adaptive enrichment of coarse components
by solving generalized eigenvalue problems.
We first give a review of the results in \cite{Pechstein:slide:2013,Klawonn:PAMM:2014,KimChung:2014} for the two-dimensional case.
We will extend the method in \cite{Klawonn:PAMM:2014} to the three-dimensional case, which will be presented in the next section.
For the two-dimensional case, an equivalence class shared by two subdomains
$\Omega_i$ and $\Omega_j$ is considered and denoted by $F$.
We note that it is identical to the face in three dimensions.
For such an equivalence class $F$, the following generalized eigenvalue problem is proposed in \cite{Klawonn:PAMM:2014}:
\begin{equation}\label{Klawonn:EIG}
\Big((D_F^{(j)})^TS_F^{(i)}D_F^{(j)} + (D_F^{(i)})^T S_F^{(j)} D_F^{(i)} \Big) v
=\lambda \, \Big( \widetilde{S}_F^{(i)}:\widetilde{S}_F^{(j)} \Big) v,
\end{equation}
where $S_F^{(i)}$ and $D_F^{(i)}$ are the block matrices of $S^{(i)}$ and $D_i$ corresponding to unknowns
interior to $F$, respectively.
The matrix $\widetilde{S}_F^{(i)}$ are the Schur complement
of $S^{(i)}$ obtained after eliminating unknowns except those interior to $F$. In addition, for symmetric and semi-positive definite matrices $A$ and $B$,
their parallel sum $A:B$ is defined by, see \cite{Parallel:Sum},
\begin{equation}
A:B=B(A+B)^{+}A, \label{eq:psum}
\end{equation}
where $(A+B)^{+}$ is a pseudo inverse of $A+B$.
We note that the problem in \eqref{Klawonn:EIG} is identical
to that considered in \cite{Pechstein:slide:2013} when $D_F^{(i)}$ are chosen as
the deluxe scalings, i.e.,
$$D_F^{(i)}=(S_F^{(i)}+S_F^{(j)})^{-1} S_F^{(i)}.$$
In our previous work by the first and second authors~\cite{KimChung:2014}, two types of generalized eigenvalue problems are considered on each $F$,
\begin{align}
\widetilde{S}_F^{(i)} v &= \lambda \widetilde{S}_F^{(j)} v,\label{KimChung:EIG:1}\\
\Big( S_F^{(i)}+S_F^{(j)} \Big) v &= \lambda \Big(\widetilde{S}_F^{(i)} + \widetilde{S}_F^{(j)} \Big)v.\label{KimChung:EIG:2}
\end{align}
We notice that, when $S_F^{(i)}=S_F^{(j)}$ and $\widetilde{S}_F^{(i)}=\widetilde{S}_F^{(j)}$
no eigenvectors will be selected from the generalized eigenvalue problem (\ref{KimChung:EIG:1}), the deluxe scaling is identical
to the multiplicity scaling, and \eqref{KimChung:EIG:2} is identical to \eqref{Klawonn:EIG}.
Thus, the method in \cite{KimChung:2014} is identical to those in \cite{Pechstein:slide:2013,Klawonn:PAMM:2014}
for the special case.

\subsection{\label{CB}Parallel sum and change of basis formulation}
In this section, we present some important properties of the parallel sum
and an upper bound estimate in Lemma~\ref{lem:c1}, which will be useful in the analysis of condition numbers.

The parallel sum (\ref{eq:psum}) has the following properties:
\begin{align}
A:B&=B:A, \label{psum:sym}\\
A:B&\le A,\quad A:B \le B. \label{psum:ineq}
\end{align}
When $A$ and $B$ are symmetric and positive definite, we have
\begin{equation}\label{psum:id} A:B=(A^{-1}+B^{-1})^{-1}.\end{equation}
For the proofs of \eqref{psum:sym}-\eqref{psum:id}, we refer to \cite{Parallel:Sum}.

In the following, we provide a key estimate in the analysis of the upper bound of condition number
in the two-dimensional case. This result will also be used in the estimate of upper bound
for face in the three-dimensional case.
Let $F$ be an equivalence class shared by two subdomains $\Omega_i$ and $\Omega_j$.
We consider the generalized eigenvalue problem proposed in \cite{Klawonn:PAMM:2014}:

\noindent
{\bf Generalized eigenvalue problem for an equivalence class shared by two subdomains}
\begin{equation}\label{Geig:pb}
A_F v=\lambda \, \widetilde{S}_F^{(i)}:\widetilde{S}_F^{(j)} v,
\end{equation}
where
$A_F=(D_F^{(j)})^T S_F^{(i)} D_F^{(j)}+(D_F^{(i)})^T S_F^{(j)} D_F^{(i)}$
and
$\widetilde{S}_F^{(l)}$, $l=i,j,$ denote the Schur complement matrix of $S^{(l)}$ after eliminating
unknowns except those interior to $F$.

We note that the eigenvalue $\lambda$ lies in the range $(0,\infty]$.
Let $\lambda_l$ be the $l$-th eigenvalue and $v_l$ be the associated eigenvector,
which is normalized with respect to the inner product $\langle A_F \, \cdot, \,\cdot \rangle$.
Let $\lambda_{TOL}$ be a given tolerance, and
assume that $\lambda_1 \ge \lambda_2 \cdots \ge \lambda_{k} \ge \lambda_{TOL} > \lambda_{k+1}$.
We enforce the following constraints on $w^{(i)}-w^{(j)}$:
\begin{equation}\label{Face-P}\langle A_F (w_F^{(i)}-w_F^{(j)}), v_l \rangle=0,\quad l=1,\cdots,k,\end{equation}
where $w_F^{(i)}$ denotes the unknowns of $w^{(i)}$ interior to $F$.
In other words, $\{v_l \}_{l=1}^k$ forms a basis of the coarse components of $w_F^{(i)}$ and $w_F^{(j)}$ over $F$, that is,
\begin{equation}
(w_F^{(i)})_{\Pi}=\sum_{l=1}^k \langle A_F w_F^{(i)},v_l \rangle v_l
\label{eq:wpi}
\end{equation}
and their coarse components on $F$ are identical, that is,
$$(w_F^{(i)})_{\Pi}=(w_F^{(j)})_{\Pi}.$$

Next, we state and prove the following result, which is useful in the analysis of an upper bound of condition numbers.
\begin{lemma}\label{lem:c1}
The coarse component defined in (\ref{eq:wpi}) satisfies
\begin{equation}
\langle A_F (w_F^{(i)}-(w_F^{(i)})_{\Pi}), w_F^{(i)}-(w_F^{(i)})_{\Pi} \rangle \le
\lambda_{TOL} \langle S^{(i)} w^{(i)},w^{(i)} \rangle. \label{upper-bound-face}
\end{equation}
\end{lemma}
\begin{proof}
First, by the generalized eigenvalue problem (\ref{Geig:pb}), we have
\begin{equation*}
\langle A_F (w_F^{(i)}-(w_F^{(i)})_{\Pi}), w_F^{(i)}-(w_F^{(i)})_{\Pi} \rangle \le
\lambda_{TOL} \langle \widetilde{S}_F^{(i)}: \widetilde{S}_F^{(j)} (w_F^{(i)}-(w_F^{(i)})_{\Pi}), w_F^{(i)}-(w_F^{(i)})_{\Pi} \rangle.
\end{equation*}
Using the orthogonality of eigenfunctions,
\begin{equation*}
\begin{split}
&\; \langle \widetilde{S}_F^{(i)}: \widetilde{S}_F^{(j)} (w_F^{(i)}-(w_F^{(i)})_{\Pi}), w_F^{(i)}-(w_F^{(i)})_{\Pi} \rangle \\
= &\; \langle \widetilde{S}_F^{(i)}: \widetilde{S}_F^{(j)} w_F^{(i)},w_F^{(i)} \rangle
- \langle \widetilde{S}_F^{(i)}:\widetilde{S}_F^{(j)} (w_F^{(i)})_{\Pi}, (w_F^{(i)})_{\Pi} \rangle.
\end{split}
\end{equation*}
Hence,
\begin{align}
 \langle A_F (w_F^{(i)}-(w_F^{(i)})_{\Pi}), w_F^{(i)}-(w_F^{(i)})_{\Pi} \rangle \nonumber
\le &\;  \lambda_{TOL} \langle \widetilde{S}_F^{(i)}: \widetilde{S}_F^{(j)} w_F^{(i)},w_F^{(i)} \rangle
\nonumber \\
\le & \; \lambda_{TOL} \langle \widetilde{S}_F^{(i)} w_F^{(i)},w_F^{(i)} \rangle \nonumber \\
\le & \; \lambda_{TOL} \langle S^{(i)} w^{(i)},w^{(i)} \rangle \nonumber
\end{align}
where (\ref{psum:ineq}) is used in the second inequality above, and the third inequality follows from the definition
of $\widetilde{S}_F^{(i)}$.
This proves the lemma.
\end{proof}



Finally,
we will present how to \revHH{apply the scaling matrices to the residual vector} when general scaling matrices are used.
In practice, we use transformed matrices in the implementation of the BDDC algorithm after choosing
the set of adaptive primal unknowns.
Let $P_F$ be the change of basis matrix of the form,
$$P_F=[ v_1 \; v_2 \; \cdots  v_k \; v_{k+1} \; \cdots \; v_{N(F)} ],$$
where $v_l$ are eigenvectors of the generalized eigenvalue problem in \eqref{Geig:pb}, and
the first $k$ vectors are related to the primal unknowns.
Let $\check{S}_F^{(i)}=P_F^T S_F^{(i)} P_F$
be the transformed matrix.
In the BDDC algorithm, after performing the change of basis, the following deluxe scaling matrices can be chosen
\begin{equation}\label{CHB:scaling}\check{D}_{F}^{(i)}=(\check{S}_{F}^{(i)} + \check{S}_{F}^{(j)})^{-1} \check{S}_{F}^{(i)}\end{equation}
and they are partitioned  into the following block form:
$$\check{D}_F^{(i)}=\begin{pmatrix} \check{D}_{F,\Pi \Pi}^{(i)} & \check{D}_{F,\Pi \Delta}^{(i)} \\
\check{D}_{F,\Delta \Pi}^{(i)} & \check{D}_{F,\Delta \Delta}^{(i)} \end{pmatrix},$$
where $\Pi$ and $\Delta$ denote blocks to the adaptive primal unknowns and the remaining dual unknowns, respectively.
In the preconditioner $M^{-1}_{BDDC}$, application of $\widetilde{D}^T$ and $\widetilde{D}$
should address the above block structure of $\check{D}_F^{(i)}$.
In detail, the result of scaling $\widetilde{D}^T$ on dual unknowns of $F$ is
$$\widetilde{D}^T \begin{pmatrix} u_{\Delta} \\ u_{\Pi}
\end{pmatrix}|_{i,F,\Delta}=(\check{D}^{(i)}_{F,\Delta \Delta})^T u_{F,\Delta}^{(i)} + {(\check{D}^{(i)}_{F,\Pi \Delta})^T u_{F,\Pi}},$$
where $u_{F,\Delta}^{(i)}$ and $u_{F,\Pi}$ are dual and primal unknowns of $F$.
The result of scaling $\widetilde{D}$ by dual unknowns $u_{\Delta}$ on unknowns of $F$
is
$$\widetilde{D} \begin{pmatrix} u_{\Delta} \\ 0 \end{pmatrix}|_{i,F}
=\begin{pmatrix}
  \check{D}_{F,\Delta \Delta}^{(i)} u_{F,\Delta}^{(i)}\\
  {  \check{D}_{F,\Pi \Delta}^{(i)} u_{F,\Delta}^{(i)}+\check{D}_{F,\Pi \Delta}^{(j)} u_{F,\Delta}^{(j)} }
\end{pmatrix}.$$

\section{\label{ext:3D}Three-dimensional extension}
For the three-dimensional case, there are three types of equivalence classes: vertices, edges, and faces.
We will choose the unknowns at subdomain vertices as part of the primal unknowns and enforce strong continuity on them.
For faces and edges, we will choose primal constraints by solving appropriate generalized eigenvalue problems.

In this section we will construct two types of algorithms and later show that
the condition numbers for both of them are bounded above by $C \lambda_{TOL}$.
\revHH{The first method is the BDDC algorithm with a change of basis formulation
and the second method is the FETI-DP algorithm with a projector preconditioner.}
The analysis of these methods will be given in Section~\ref{analysis}.


The adaptive primal constraints on faces can be chosen similarly as in the two-dimensional
case by solving the generalized eigenvalue problem in \eqref{Geig:pb}.
On the other hand, edges in the three-dimensional case are equivalence classes shared by more than two subdomains and thus construction of appropriate generalized eigenvalue problems
has been a difficult task. In the following we present an elaborate construction of generalized eigenvalue problems
for the edges.

\subsection{\label{sec:GEIG:edge}Generalized eigenvalue problem for edge}
To simplify the notations, we assume that an edge $E$ is shared by three subdomains $\{\Omega_i,\Omega_j,\Omega_k\}$.
We will see later that,
in order to derive an upper bound for the condition number, we have to estimate the following term
on the edge $E$ for the subdomain $\Omega_i$,
\begin{equation}\label{edge-term}
\sum_{l=j,k} \langle (D_E^{(l)})^T S_E^{(i)} D_E^{(l)}(w_E^{(i)}-w_E^{(l)}), w_E^{(i)}-w_E^{(l)} \rangle,
\end{equation}
where $S_E^{(i)}$ is the block matrix of $S^{(i)}$ corresponding to unknowns interior to the edge $E$
and $D_E^{(l)}$ are weight factors introduced in the preconditioner.
We remark that one needs to estimate
similar terms for the other subdomains $\Omega_j$ and $\Omega_k$ sharing the edge $E$,
\begin{align}
\sum_{l=i,k} \langle (D_E^{(l)})^T S_E^{(j)} D_E^{(l)}(w_E^{(j)}-w_E^{(l)}), w_E^{(j)}-w_E^{(l)} \rangle,
\label{edge-term-j}\\
\sum_{l=i,j} \langle (D_E^{(l)})^T S_E^{(k)} D_E^{(l)}(w_E^{(k)}-w_E^{(l)}), w_E^{(k)}-w_E^{(l)} \rangle.
\label{edge-term-k}
\end{align}

By subtracting the common primal part $(w_E^{(l)})_{\Pi}$ from each $w_E^{(l)}$, $l=i,j,k$ and collecting terms
for each $z_E^{(i)}=w_E^{(i)}-(w_E^{(i)})_{\Pi}$ from \eqref{edge-term}-\eqref{edge-term-k},
\begin{equation*}
\langle \sum_{l=j,k}( (D_E^{(l)})^T S_E^{(i)} D_E^{(l)}+(D_E^{(i)})^T S_E^{(l)}D_E^{(i)}) z_E^{(i)},z_E^{(i)} \rangle,
\end{equation*}
we can bound the sum of the terms~\eqref{edge-term}-\eqref{edge-term-k}
by the following
\begin{equation}\label{three-terms}
2 \left( \langle A_E^{(i)} z_E^{(i)},z_E^{(i)}\rangle
+\langle A_E^{(j)} z_E^{(j)},z_E^{(j)}\rangle
+\langle A_E^{(k)} z_E^{(k)},z_E^{(k)}\rangle\right),
\end{equation}
where
\begin{equation}
A_E^{(m)}=\sum_{l \in I(E)\setminus\{m\}} ( (D_E^{(l)})^T S_E^{(m)} D_E^{(l)} + (D_E^{(m)})^T S_E^{(l)}D_E^{(m)} ),
\label{eq:AE}
\end{equation}
and $I(E)$ is the set of subdomain indices sharing the edge $E$.

Based on this observation, we give generalized eigenvalue problems for edges:

\noindent
{\bf Generalized eigenvalue problem for edge}
\begin{equation}
A_E v=\lambda \widetilde{S}_E v,
\label{eq:Geig:edge}
\end{equation}
where $$A_E=\sum_{m \in I(E)} \sum_{l \in I(E)\setminus\{m\}} (D_E^{(l)})^T S_E^{(m)} D_E^{(l)} ,\quad \widetilde{S}_E=\widetilde{S}_E^{(i)}:\widetilde{S}_E^{(j)}:\widetilde{S}_E^{(k)},$$
and $\widetilde{S}_E^{(l)}$ are the Schur complement obtained from $S^{(l)}$ after eliminating
unknowns except those interior to $E$.
We notice that
\begin{equation}
\label{eq:matrixineq}
A_E^{(m)} \le A_E, \quad  \widetilde{S}_E \le \widetilde{S}_E^{(m)}, \quad \forall m \in I(E).
\end{equation}
We choose the eigenvectors $v_l$ with their associated eigenvalues $\lambda_l$
greater than $\lambda_{TOL}$ as primal components.
The eigenvectors are then normalized with respect to the inner product,
$\langle A_E \, \cdot \, , \, \cdot \rangle$.
Let $w_E^{(l)}$ denote the unknowns of $w^{(l)}$ interior to the edge $E$.
With the chosen set of eigenvectors $\{ v_n \}_{n=1}^{P_E}$, we enforce the following
primal constraints on $w_E^{(l)}$,
\begin{equation}\label{Edge-P}\langle A_E v_n,(w_E^{(l)}-w_E^{(k)}) \rangle=0,\quad n=1,2,\cdots, P_E.
\end{equation}
In other words, the coarse components are
\begin{equation}
(w_E^{(l)})_{\Pi}=\sum_{n=1}^{P_E}\langle A_E v_n, w_E^{(l)} \rangle v_n, \quad l=i,j,k.
\label{eq:edgecoarse}
\end{equation}
Next, we prove the following inequality, which is crucial in our analysis of condition number bounds.

\begin{lemma}\label{edge:upperbd}
For the coarse component defined in (\ref{eq:edgecoarse}), we obtain that
\begin{equation}
\langle A_E^{(i)} z_E^{(i)},z_E^{(i)} \rangle
\le
\lambda_{TOL} \langle S^{(i)} w^{(i)},w^{(i)} \rangle, \label{3d-edge-bound}
\end{equation}
where $z_E^{(i)}=w_E^{(i)}-(w_E^{(i)})_{\Pi}$ and  $A_E^{(i)}$ is given in (\ref{eq:AE}).
\end{lemma}
\begin{proof}
The proof is similar to that of Lemma \ref{lem:c1}. By using (\ref{eq:matrixineq}) and the
generalized eigenvalue problem (\ref{eq:Geig:edge}), we have
\begin{equation*}
\langle A_E^{(i)} z_E^{(i)},z_E^{(i)} \rangle  \le \langle A_E z_E^{(i)},z_E^{(i)} \rangle
\le \lambda_{TOL} \langle \widetilde{S}_E z_E^{(i)},z_E^{(i)} \rangle.
\end{equation*}
Using the definition of $z_E^{(i)}$ and the orthogonality of eigenfunctions, we obtain
\begin{equation*}
\langle \widetilde{S}_E z_E^{(i)},z_E^{(i)} \rangle
\le
\left(\langle \widetilde{S}_E w_E^{(i)},w_E^{(i)} \rangle-
\langle \widetilde{S}_E (w_E^{(i)})_{\Pi}, (w_E^{(i)})_{\Pi} \rangle \right).
\end{equation*}
Combining the above results,
\begin{align*}
\langle A_E^{(i)} z_E^{(i)},z_E^{(i)} \rangle
& \le \lambda_{TOL} \langle \widetilde{S}_E w_E^{(i)},w_E^{(i)} \rangle \nonumber \\
& \le \lambda_{TOL} \langle \widetilde{S}_E^{(i)} w_E^{(i)},w_E^{(i)} \rangle \\
 & \le \lambda_{TOL} \langle S^{(i)} w^{(i)},w^{(i)} \rangle \nonumber
\end{align*}
where we have used (\ref{eq:matrixineq}) in the second inequality and the definition of $\widetilde{S}_E^{(i)}$
in the last inequality.
\end{proof}

Finally, we remark that the above result holds for
the other terms in (\ref{three-terms}).

Here, we will present how to \revHH{apply} the scaling matrices on edges when general scaling matrices are used.
Similar to $P_F$, let $P_E$ be the change of basis matrix of the form,
$$P_E=[ v_1 \; v_2 \; \cdots  v_{P_E} \; v_{P_E+1} \; \cdots \; v_{N(E)} ],$$
where $N(E)$ is the number of interior unknowns on $E$, $v_l$ are eigenvectors of the generalized eigenvalue problem in \eqref{eq:Geig:edge}, and
the first $P_E$ vectors are related to the primal unknowns.
Let $\check{S}_E^{(i)}=P_E^T S_E^{(i)} P_E$
be the transformed matrix.
In the BDDC algorithm, after performing the change of basis, the following deluxe scaling matrices can be chosen
\begin{equation}\label{CHB:scaling:edge}
\check{D}_{E}^{(i)}=(\check{S}_{E}^{(i)} + \check{S}_{E}^{(j)} + \check{S}_{E}^{(k)})^{-1} \check{S}_{E}^{(i)}
\end{equation}
and they are partitioned  in the following block form:
$$\check{D}_E^{(i)}=\begin{pmatrix}
\check{D}_{E,\Pi \Pi}^{(i)} & \check{D}_{E,\Pi \Delta}^{(i)} \\
\check{D}_{E,\Delta \Pi}^{(i)} & \check{D}_{E,\Delta \Delta}^{(i)} \end{pmatrix},$$
where $\Pi$ and $\Delta$ denote blocks to the adaptive primal unknowns and the remaining dual unknowns, respectively.
In the preconditioner $M^{-1}_{BDDC}$, application of $\widetilde{D}^T$ and $\widetilde{D}$
should address the above block structure of $\check{D}_E^{(i)}$ as in $\check{D}_F^{(i)}$.
In detail, the result of scaling $\widetilde{D}^T$ on dual unknowns of $E$ is
$$\widetilde{D}^T \begin{pmatrix} u_{\Delta} \\ u_{\Pi}
\end{pmatrix}|_{i,E,\Delta}=(\check{D}^{(i)}_{E,\Delta \Delta})^T u_{E,\Delta}^{(i)} + {(\check{D}^{(i)}_{E,\Pi \Delta})^T u_{E,\Pi}},$$
where $u_{E,\Delta}^{(i)}$ and $u_{E,\Pi}$ are dual and primal unknowns of $E$ in $\Omega_i$, respectively.
The result of scaling $\widetilde{D}$ by dual unknowns $u_{\Delta}$ on unknowns of $E$
is
$$\widetilde{D} \begin{pmatrix} u_{\Delta} \\ 0 \end{pmatrix}|_{i,E}
=\begin{pmatrix}
  \check{D}_{E,\Delta \Delta}^{(i)} u_{E,\Delta}^{(i)}\\
  { \sum_{m \in I(E)} \check{D}_{E,\Pi \Delta}^{(m)} u_{E,\Delta}^{(m)} }
\end{pmatrix}.$$

\subsection{\label{3d-constraints-weights} Constraints and scaling matrices in FETI-DP methods}
In this section, we discuss constraints and scaling matrices used for the FETI-DP methods in
the three-dimensional case.
One key feature of our method is that
we will enforce the continuity on edges $E$ using
fully redundant Lagrange multipliers
and use
appropriate scaling matrices for each of the Lagrange multipliers.
For simplicity, we
assume that there are three subdomains $\{ \Omega_i, \Omega_j,\Omega_k\}$ sharing an edge $E$.
We use
$u_{E,l}^{(i)}$, $u_{E,l}^{(j)}$, and $u_{E,l}^{(k)}$ to represent the
unknowns sharing the same geometric location on $E$.
We then enforce the fully redundant continuity constraints on these unknowns as follows:
\begin{align*}
u_{E,l}^{(i)}-u_{E,l}^{(j)}&=0,\; l=1,\cdots, N_E, \\
u_{E,l}^{(i)}-u_{E,l}^{(k)}&=0, \; l=1,\cdots, N_E,\\
u_{E,l}^{(j)}-u_{E,l}^{(k)}&=0, \; l=1,\cdots, N_E.
\end{align*}
The above constraints will be enforced weakly using Lagrange multipliers.
Because of this construction, the formulation of the FETI-DP method in the three-dimensional case
is similar to the one presented in Section \ref{BDDC:FETI:theory} for the two-dimensional case.

Similar to the discussion in Section \ref{BDDC:FETI:theory},
we need to define the matrix $B_{D,\Delta}^{(i)}$; see \eqref{BTDs}.
In three dimensions, the rows corresponding to Lagrange multipliers to the unknowns on a face $F$
are multiplied with the scaling matrix $(D_{F}^{(j)})^T$ as in the two dimensions.
On the other hand, the unknowns on an edge $E$ are shared by more than two subdomains, for example, $\Omega_l$, $l=i,j,k$, and
thus the rows corresponding to Lagrange multipliers connecting $\Omega_i$ and $\Omega_l$
are multiplied with the scaling matrix $(D_{E}^{(l)})^T$, $l=j,k$.
We note that  the scaling matrices $D_E^{(l)}$ satisfy $\sum_{m=i,j,k} D_E^{(m)}=I$.
Finally,
we remark that for the use of simple multiplicity scalings, the above matrices are given as
$D_E^{(m)}=1/3$ and $D_F^{(m)}=1/2$, and for the deluxe scalings $D_E^{(m)}=(S_E^{(i)}+S_E^{(j)}+S_E^{(k)})^{-1} S_E^{(m)}$ with $m=i,j,k$ and
$D_F^{(m)}=(S_F^{(i)}+S_F^{(j)})^{-1} S_F^{(m)}$ with $m=i,j$.

\section{\label{analysis} Estimates of condition numbers}
In this section, we will derive upper bounds of the condition numbers for the two algorithms considered in this paper.
In particular, we will show that, for the BDDC algorithm with a change
of basis formulation and for the FETI-DP
algorithm with a projector preconditioner, the condition numbers
are bounded above by $C \lambda_{TOL}$.
We note that the adaptive primal constraints are enforced by a projection in the later case.

We will now form a matrix $U$ by using all the chosen eigenvectors
from the generalized eigenvalue problems (\ref{Geig:pb}) and (\ref{eq:Geig:edge}).
The matrix $U$ will be used in the projector preconditioner to enforce the adaptively
chosen primal constraints in \eqref{Face-P} and \eqref{Edge-P} on the residual at each iteration
of the FETI-DP algorithm.
Notice that the dimension of the matrix $U$ is $N_M \times N_P$,
where $N_M$ is the number of degrees of freedom of the Lagrange multipliers,
and $N_P$ is obtained by adding the following number from all edges $E$ and all faces $F$:
the number of selected eigenvectors per an edge $E$ (or a face $F$) times the number of all possible pairs among the subdomains sharing $E$ (or $F$).
In a more detail, for an edge $E$ with the chosen set of eigenvectors $\{ v_1,\cdots,v_{P_E} \}$
the following primal constraints are enforced on $E$,
$$\langle A_E v_n, w_{E}^{(l)}-w_{E}^{(k)} \rangle=0,\quad \forall l,k \in I(E),\quad \forall n=1,\cdots,P_E,$$
and thus the matrix $U$ can be formed from $A_E v_n$ by putting $A_E v_n$ to the corresponding rows and columns in the matrix $U$.

For the FETI-DP method with a projector preconditioner, one first formulates the standard FETI-DP
method and solves the corresponding algebraic system by a projector preconditioner; see~\cite{PP-FETI-DP}
for more details.
In the standard FETI-DP method, the unknowns at subdomain vertices are selected as primal unknowns
and strong continuity is enforced on them, and Lagrange multipliers are introduced
to enforce weak continuity on the remaining unknowns of the subdomain interfaces.
We note that for an easier implementation of $3D$ problems, we have employed fully redundant Lagrange multipliers.
As a result of elimination of unknowns except the Lagrange multipliers, \revHH{the dual problem in \eqref{sys:Fdp}}
is obtained for $\lambda \in \text{Range}(B)$.
Let
\begin{equation}
F_{DP}=B \widetilde{S}^{-1} B^T. \label{F_DP}
\end{equation}
We introduce the following projection
\begin{equation*}\label{def:P}
P=U(U^T F_{DP} U)^{-1} U^T F_{DP},
\end{equation*}
and it satisfies the following properties, see \cite{PP-FETI-DP}:
\begin{align}
\text{Range}(I-P)   &\perp_{F_{DP}} \text{Range}(P)=\text{Range}(U), \label{propU1}\\
\text{Range}(I-P^T) &\perp  \; \; \; \text{Ker}(I-P)=\text{Range}(U). \label{propU2}
\end{align}

The FETI-DP algorithm with a projector preconditioner solves \revHH{\eqref{sys:Fdp}} by a preconditioned conjugate gradient method
using the following preconditioner
\begin{equation}
M^{-1}_{PP}=P_B^T(I-P) M_{FETI}^{-1} (I-P)^T P_B, \label{M_PP}
\end{equation}
where $P_B$ is the orthogonal projection onto $\text{Range}(B)$ and \revHH{$M_{FETI}^{-1}$ is defined in \eqref{precond:Fdp}.}
In practice, $M_{PP}^{-1}$ is applied to the residual vector at each iteration
and thus the projections $P_B$ and $P_B^T$ in $M_{PP}^{-1}$ do not need to be enforced
due to the operators $B$ and $B^T$ in $F_{DP}$.


\subsection{Change of basis formulation}

Now, we will derive an upper bound of the condition number
for the case of change of basis formulation.
We can form the resulting BDDC system after performing a change of basis to make the adaptive primal constraints explicit and treat them just like unknowns at subdomain vertices.
To stress this, we will use the notation $\widetilde{S}_{a}$ instead of the standard notation $\widetilde{S}$ and similarly $\widetilde{w}_a$ for unknowns $\widetilde{w}$.
\revHH{Following the proof in \cite[Theorem 1]{JW-Stokes-BDDC-Proof}, we will need to estimate the bound,
$$\langle \widetilde{S}_a (I-E_D) \widetilde{w}_a, (I-E_D) \widetilde{w}_a \rangle \le C
\langle \widetilde{S}_a \widetilde{w}_a, \widetilde{w}_a \rangle.$$}
We also note that Lemmas~\ref{lem:c1} and \ref{edge:upperbd} hold for the matrices after the change of basis.

\begin{lemma}\label{upper-bound-change-of-basis}
We obtain
$$\langle \widetilde{S}_a (I-E_D) \widetilde{w}_a, (I-E_D) \widetilde{w}_a \rangle \le C \lambda_{TOL}
\langle \widetilde{S}_a \widetilde{w}_a, \widetilde{w}_a \rangle,$$
where $\lambda_{TOL}$ is the tolerance used in the selection of adaptive primal constraints
and $C$ is a constant depending only on $N_{F(i)}$, $N_{E(i)}$, $N_{I(E)}$, which are the number of faces per subdomain,
the number of edges per subdomain, and the number of subdomains sharing an edge $E$, respectively. In particular,
\begin{equation}\label{explicit:C} C=8 \max \left\{  \max_i \{ N_{F(i)}^2  \},\; \max_i  \left\{ N_{E(i)}^2 \max_{E \in E(i)} \{ N_{I(E)} \}  \right\} \right\},\end{equation}
where $E(i)$ is the set of edges in $\partial \Omega_i$.
\end{lemma}
\begin{proof}
Let
$z^{(i)}=((I-E_D) \widetilde{w}_a)|_{\partial \Omega_i}$.
Since $\widetilde{w}_a$ is continuous at the primal unknowns,
$$z^{(i)}_{F}=D_{F}^{(j)} (\widetilde{w}_{F,\Delta}^{(i)}-\widetilde{w}_{F,\Delta}^{(j)})$$
and
$$z^{(i)}_{E}=\sum_{m \in I(E) \setminus \{i\}} D_{E}^{(m)} (\widetilde{w}_{E,\Delta}^{(i)}-\widetilde{w}_{E,\Delta}^{(m)}),$$
where $F$ is a face common to $\Omega_i$ and $\Omega_j$, and $E$ is an edge common to $\Omega_i$ and $\Omega_m$
with $m \in I(E) \setminus \{i\}$. Recall that $I(E)$ is the set of subdomain indices sharing
the edge $E$. In the above, $z^{(i)}_F$ are restriction of $z^{(i)}$ to the unknowns in $F$, and $\widetilde{w}^{(i)}_{F,\Delta}$ consist of
the dual unknowns of $\widetilde{w}_a$ interior to $F \bigcap \partial\Omega_i$ and the zero primal unknowns.
The definitions for $z_E^{(i)}$ and $\widetilde{w}_{E,\Delta}^{(i)}$ are similar.

We then obtain
\begin{align}
&\langle \widetilde{S}_a (I-E_D) \widetilde{w}_a, (I-E_D) \widetilde{w}_a \rangle =\sum_{i=1}^N
\langle S^{(i)} z^{(i)}, z^{(i)} \rangle  \nonumber \\
&\le \sum_{i=1}^N \left( 2 N_{F(i)}\sum_{F \in F(i)} \langle S^{(i)}_{F} z_{F}^{(i)}, z_{F}^{(i)} \rangle
+ 2 N_{E(i)} \sum_{E \in E(i)} \langle S^{(i)}_{E} z^{(i)}_{E},z^{(i)}_{E} \rangle\right), \label{bd1:two-terms}
\end{align}
where $F(i)$ and $E(i)$ denote the set of faces and edges in $\Omega_i$, respectively,
and $N_{F(i)}$ and $N_{E(i)}$ denote the number of faces and edges in $F(i)$ and $E(i)$, respectively.

The first term in (\ref{bd1:two-terms}) can be estimated in the following way
\begin{align*}
&\;\sum_{i=1}^N 2 N_{F(i)}\sum_{F \in F(i)} \langle S^{(i)}_{F} z^{(i)}_{F}, z^{(i)}_{F} \rangle \nonumber \\
\le &\; 2 \max_i \{ N_{F(i)}\} \sum_{i=1}^N \sum_{F \in F(i)} \langle S_{F}^{(i)} D_{F}^{(j)} (\widetilde{w}_{F,\Delta}^{(i)}-\widetilde{w}_{F,\Delta}^{(j)}),D_{F}^{(j)} (\widetilde{w}_{F,\Delta}^{(i)}-\widetilde{w}_{F,\Delta}^{(j)})\rangle \nonumber \\
 \le &\; 4 \max_i \{ N_{F(i)}\} \sum_{i=1}^N \sum_{F \in F(i)} \langle A_F \widetilde{w}_{F,\Delta}^{(i)},\widetilde{w}_{F,\Delta}^{(i)}\rangle,\nonumber
\end{align*}
where
$A_F=(D_{F}^{(j)})^T S_{F}^{(i)} D_{F}^{(j)}
+(D_{F}^{(i)})^T S_{F}^{(j)} D_{F}^{(i)}$
and we have collected terms for $\widetilde{w}_{F,\Delta}^{(i)}$ in the last inequality.
By Lemma~\ref{lem:c1}, we obtain that
\begin{equation}\label{result:bd:first-term}
\sum_{i=1}^N 2 N_{F(i)}\sum_{F \in F(i)} \langle S^{(i)}_{F} z^{(i)}_{F}, z^{(i)}_{F} \rangle
\le 4 (\max_i \{ N_{F(i)} \})^2 \lambda_{TOL} \langle \widetilde{S}_a \widetilde{w}_a, \widetilde{w}_a \rangle.
\end{equation}

We now consider the second term in \eqref{bd1:two-terms}. First we have
\begin{align}
&\sum_{i=1}^N 2 N_{E(i)} \sum_{E \in E(i)} \langle S^{(i)}_{E} z^{(i)}_{E},z^{(i)}_{E} \rangle \nonumber \\
&=\sum_{i=1}^N 2 N_{E(i)} \sum_{E \in E(i)} \langle S^{(i)}_{E} \sum_{k \in I(E)\backslash \{i\}} D_{E}^{(k)} (\widetilde{w}_{E,\Delta}^{(i)}-\widetilde{w}_{E,\Delta}^{(k)}), \sum_{k \in I(E)\backslash \{i\}} D_{E}^{(k)} (\widetilde{w}_{E,\Delta}^{(i)}-\widetilde{w}_{E,\Delta}^{(k)})\rangle, \nonumber
\end{align}
which implies
\begin{equation}
\begin{split}
&\;\sum_{i=1}^N 2 N_{E(i)} \sum_{E \in E(i)} \langle S^{(i)}_{E} z^{(i)}_{E},z^{(i)}_{E} \rangle  \\
\le &\; \sum_{i=1}^N 2 N_{E(i)} \sum_{E \in E(i)}  N_{I(E)}  \sum_{k \in I(E) \backslash\{i\}} \langle S^{(i)}_{E}
D_{E}^{(k)} (\widetilde{w}_{E,\Delta}^{(i)}-\widetilde{w}_{E,\Delta}^{(k)}),
D_{E}^{(k)} (\widetilde{w}_{E,\Delta}^{(i)}-\widetilde{w}_{E,\Delta}^{(k)}) \rangle. \label{bd1:second-term}
\end{split}
\end{equation}
We notice that the last sum in (\ref{bd1:second-term}) can be estimated in the following way
\begin{align}
& \sum_{k \in I(E) \backslash\{i\}} \langle S^{(i)}_{E}
D_{E}^{(k)} (\widetilde{w}_{E,\Delta}^{(i)}-\widetilde{w}_{E,\Delta}^{(k)}),
D_{E}^{(k)} (\widetilde{w}_{E,\Delta}^{(i)}-\widetilde{w}_{E,\Delta}^{(k)}) \rangle \nonumber \\
& \le 2 \sum_{k \in I(E)\setminus\{i\}} \left( \langle S^{(i)}_{E}
D_{E}^{(k)} \widetilde{w}_{E,\Delta}^{(i)},D_{E}^{(k)} \widetilde{w}_{E,\Delta}^{(i)}\rangle
+ \langle S^{(i)}_{E}
D_{E}^{(k)} \widetilde{w}_{E,\Delta}^{(k)},D_{E}^{(k)} \widetilde{w}_{E,\Delta}^{(k)}\rangle \right).\nonumber
\end{align}
Using this relation and \eqref{bd1:second-term}, we obtain
\begin{align}
&\sum_{i=1}^N 2 N_{E(i)} \sum_{E \in E(i)} \langle S^{(i)}_{E} z^{(i)}_{E},z^{(i)}_{E} \rangle \nonumber \\
\le &4 \max_i \{ N_{E(i)} \max_{E \in E(i)} \{ N_{I(E)} \} \}  \sum_{i=1}^N \sum_{E \in E(i)} \langle A_E^{(i)} \widetilde{w}_{E,\Delta}^{(i)},\widetilde{w}_{E,\Delta}^{(i)}\rangle, \nonumber
\end{align}
where $A_E^{(i)}=\sum_{k \in I(E) \setminus \{i\} } ( (D_{E}^{(k)})^T S_{E}^{(i)} D_{E}^{(k)}+  (D_{E}^{(i)})^T S_{E}^{(k)} D_{E}^{(i)}   )$
and we have collected terms for $\widetilde{w}_{E,\Delta}^{(i)}$ in the last inequality.
By Lemma~\ref{edge:upperbd}, we thus obtain that
\begin{align}
&\sum_{i=1}^N 2 N_{E(i)} \sum_{E \in E(i)} \langle S^{(i)}_{E} z^{(i)}_{E},z^{(i)}_{E} \rangle \nonumber \\
& \le 4 \max_i \{ N_{E(i)} \max_{E \in E(i)} \{ N_{I(E)} \} \} \lambda_{TOL} \sum_{i=1}^N \sum_{E \in E(i)}
\langle S^{(i)} w^{(i)},w^{(i)} \rangle \nonumber \\
& \le 4 \max_i \{ (N_{E(i)})^2 \max_{E \in E(i)} \{ N_{I(E)} \} \} \lambda_{TOL}
\langle \widetilde{S}_a \widetilde{w}_a, \widetilde{w}_a \rangle. \label{result:bd:second-term}
\end{align}
%
Combining \eqref{bd1:two-terms} with \eqref{result:bd:first-term} and \eqref{result:bd:second-term},
the resulting bound is obtained.
\end{proof}

\revHH{From the above lemma and following similarly as in \cite[Theorem 1]{JW-Stokes-BDDC-Proof} we obtain:}
\begin{theorem}
The BDDC algorithm with a change of basis formulation for the adaptively chosen set of
primal constraints with a given tolerance $\lambda_{TOL}$ has the following bound of condition numbers,
$$\kappa(M_{BDDC,a}^{-1} \widetilde{R}^T \widetilde{S}_{a} \widetilde{R}) \le C \lambda_{TOL},$$
where $C$ is a constant depending only on $N_{F(i)}$, $N_{E(i)}$, $N_{I(E)}$, which are the number of faces per subdomain,
the number of edges per subdomain, and the number of subdomains sharing an edge $E$, respectively, with its explicit
form in \eqref{explicit:C}.
\end{theorem}

\subsection{Projector preconditioner}

We now consider the FETI-DP algorithm with
a projector preconditioner. Our goal of this section is to prove the following estimate
$$\kappa( M^{-1}_{PP} F_{DP}) \le C \lambda_{TOL},$$
where $M_{PP}^{-1}$ and $F_{DP}$ are defined in (\ref{M_PP}) and (\ref{F_DP}), respectively.
This result is stated in Theorem \ref{thm:proj}.
\revHH{We note that the proof provided in \cite{Klawonn:PAMM:2014} is limited to the multiplicity scaling; see also \cite{PP-FETI-DP}. A complete and shorter proof for arbitrary scalings is given in a more recent work~\cite{KRR-2015}. For completeness, we provide a different proof for the same result, following \cite{TW-Book}.}
Before proving this final result, we need some auxiliary lemmas.

\begin{lemma}\label{Fdp-identity}
For $\lambda$ in $\text{Range}(I-P) \bigcap \text{Range}(B)$ and $z$ in $\widetilde{W}$, the following identity holds:
$$\sup_{U^T Bz=0} \frac{ \langle \lambda, Bz \rangle^2}{\langle \widetilde{S} z, z \rangle}=\langle F_{DP} \lambda, \lambda \rangle.$$
\end{lemma}

\begin{proof}
Notice that
\begin{align*}
\sup_{U^T Bz=0} \frac{\langle \lambda, Bz \rangle^2}{\langle \widetilde{S}^{1/2}z, \widetilde{S}^{1/2}z \rangle}
&=\sup_{U^T B \widetilde{S}^{-1/2} w =0} \frac{ \langle \lambda, B\widetilde{S}^{-1/2}w \rangle^2}{\langle w, w \rangle}\\
&=\sup_{U^T B \widetilde{S}^{-1/2} w =0} \frac{ \langle \widetilde{S}^{-1/2} B^T \lambda, w \rangle^2}{\langle w, w \rangle}.
\end{align*}
Since $\lambda$ is in $\text{Range}(I-P)$, by \eqref{propU1}
$$U^T B \widetilde{S}^{-1/2} (\widetilde{S}^{-1/2} B^T \lambda)=U^T F_{DP} \lambda=0.$$
In the above, we thus choose
$w=\widetilde{S}^{-1/2} B^T \lambda$, which attains the supremum, to obtain the resulting identity:
$$\sup_{U^T Bz=0} \frac{\langle \lambda, Bz \rangle^2}{\langle \widetilde{S}^{1/2}z, \widetilde{S}^{1/2}z \rangle}
=\langle w, w \rangle=\langle B \widetilde{S}^{-1} B^T \lambda, \lambda \rangle=\langle F_{DP} \lambda, \lambda \rangle.$$
\end{proof}

Next, we recall the jump operator
$$P_D=B_D^T B,$$
and by the definitions of $B_D^T$ and $B$ it is known to preserve the jump, $Bw$, over the subdomain interfaces
$$BP_Dw=Bw.$$
In addition, $P_D$ satisfies the following result:
\begin{lemma}\label{Pd-bound}
For any $z$ in $\widetilde{W}$ such that
$U^T Bz=0$,
we have
$$ \langle \widetilde{S} P_D z, P_D z \rangle \le C \lambda_{TOL} \langle \widetilde{S} z,z \rangle,$$
where $C$ is a constant depending only on $N_{F(i)}$, $N_{E(i)}$, $N_{I(E)}$, which are the number of faces per subdomain,
the number of edges per subdomain, and the number of subdomains sharing an edge $E$, respectively.
\end{lemma}

The above lemma can be proved following the ideas in Lemma~\ref{upper-bound-change-of-basis} and
using the estimates in \eqref{upper-bound-face} and \eqref{3d-edge-bound}, and the fact that
$U^T B z=0$ implies that $z$ satisfies the constraints selected from generalized eigenvalue problems,
see \eqref{Face-P}, \eqref{Edge-P}, and the definition of $U$ in the beginning of Section~\ref{analysis}.

Now, we prove the following upper bound for $F_{DP}$.
\begin{lemma}\label{upper-bound-Fdp}(Upper bound)
For any $\lambda$ in $\text{Range}(I-P) \bigcap \text{Range}(B)$, we have
$$\langle F_{DP} \lambda , \lambda \rangle \le C \lambda_{TOL}
\sup_{U^T \mu=0,\, \mu \in \text{Range}(B)} \frac{\langle \lambda, \mu \rangle^2}{\langle M^{-1}_{FETI} \mu,\mu \rangle},$$
where $C$ is the constant described in Lemma~\ref{Pd-bound}.
\end{lemma}
\begin{proof}
Using Lemma~\ref{Fdp-identity} and Lemma~\ref{Pd-bound}, we obtain
\begin{align*}
\langle F_{DP} \lambda , \lambda \rangle & = \sup_{U^T Bz=0} \frac{ \langle \lambda, Bz \rangle^2}{\langle \widetilde{S} z, z \rangle}\\
&\le C \lambda_{TOL} \sup_{U^T Bz=0} \frac{\langle \lambda, Bz \rangle^2}{\langle \widetilde{S} P_D z, P_D z \rangle}\\
&=  C \lambda_{TOL} \sup_{U^T Bz=0} \frac{\langle \lambda, Bz \rangle^2}{\langle B_D \widetilde{S} B_D^T Bz,Bz \rangle}.
\end{align*}
Letting $\mu=Bz$ in the above expression, we have
$$\langle F_{DP} \lambda , \lambda \rangle \le C \lambda_{TOL} \sup_{U^T \mu=0,\, \mu \in \text{Range}(B)}
\frac{\langle \lambda, \mu \rangle^2}{\langle M^{-1}_{FETI} \mu,\mu \rangle}.$$
\end{proof}

\revHH{We note that $\mu$ in the above lemma is in $\text{Range}(I-P^T) \bigcap \text{Range}(B)$ by \eqref{propU2}.}
In the above Lemma~\ref{upper-bound-Fdp}, the supremum occurs when $\mu$ is chosen as $P_B^T (I-P) M^{-1} (I-P^T) P_B \mu=\lambda$, that is, $M_{PP}^{-1} \mu= \lambda$.
Recall that $$M_{PP}^{-1}=P_B^T (I-P) M^{-1}_{FETI} (I-P^T) P_B,$$
$\mu$ is in $\text{Range}(B)$, and $P_B$ is the orthogonal projection onto $\text{Range}(B)$.
We thus obtain the following relation
\begin{equation}\label{max-eigvalue}\langle F_{DP} M^{-1}_{PP} \mu , M^{-1}_{PP} \mu \rangle \le C \lambda_{TOL} \langle M^{-1}_{PP} \mu, \mu \rangle.
\end{equation}

We now prove the following lower bound for $F_{DP}$.
\begin{lemma}\label{lower-bound-Fdp}(Lower bound)
For any $\lambda$ in $\text{Range}(I-P) \bigcap \text{Range}(B)$, we have
$$\langle F_{DP} \lambda , \lambda \rangle \ge
\sup_{U^T \mu=0,\, \mu \in \text{Range}(B)} \frac{\langle \lambda, \mu \rangle^2}{\langle M^{-1}_{FETI} \mu,\mu \rangle}.$$
\end{lemma}
\begin{proof}
Similar to the proof of Lemma \ref{upper-bound-Fdp}, we notice that
\begin{align*}
\langle F_{DP} \lambda , \lambda \rangle & = \sup_{U^T Bz=0} \frac{ \langle \lambda, Bz \rangle^2}{\langle \widetilde{S} z, z \rangle}\\
& \ge \sup_{U^T B P_Dz=0} \frac{\langle \lambda, B P_D z \rangle^2}{\langle \widetilde{S} P_D z, P_D z \rangle}.
\end{align*}
Since $P_D z$ preserves the jump, that is, $B P_Dz=Bz$, we finally obtain
$$\langle F_{DP} \lambda , \lambda \rangle \ge \sup_{U^T Bz=0} \frac{\langle \lambda, B z \rangle^2}{\langle M^{-1}_{FETI} Bz, Bz \rangle}.$$
The result follows by letting $\mu=Bz$.
\end{proof}

By the above Lemma~\ref{lower-bound-Fdp}, we can obtain that
\begin{equation}\label{min-eigvalue}\langle F_{DP} M^{-1}_{PP} \mu , M^{-1}_{PP} \mu \rangle \ge \langle M^{-1}_{PP} \mu, \mu \rangle.
\end{equation}

By using Lemmas~\ref{upper-bound-Fdp} and \ref{lower-bound-Fdp}, see also \eqref{max-eigvalue} and \eqref{min-eigvalue},
we obtain the resulting condition number bound:
\begin{theorem}\label{thm:proj}
For the FETI-DP algorithm with the projector preconditioner $M_{PP}^{-1}$, we obtain
$$\kappa(M_{PP}^{-1} F_{DP}) \le C \lambda_{TOL},$$
where $C$ is a constant depending only on $N_{F(i)}$, $N_{E(i)}$, $N_{I(E)}$, which are the number of faces per subdomain,
the number of edges per subdomain, and the number of subdomains sharing an edge $E$, respectively, with its explicit
form stated in \eqref{explicit:C}.
\end{theorem}

\section{Numerical results}\label{sec:num}
In this section, we present some numerical results to show the performance of our BDDC/FETI-DP algorithm
with an enriched coarse space.
We will test our algorithm for various choices of the coefficient $\rho(x)$.
We will present numerical tests for the two-dimensional case in Section~\ref{sec:2d}
and three-dimensional case in Section~\ref{sec:3d}.

\subsection{$2D$ case}\label{sec:2d}
In our experiments, we consider a unit square domain $\Omega$
and partition it into uniform square subdomains.
Each subdomain is then divided into uniform grids with a
grid size $h$. We use $H$ to denote the size of the subdomains.
In the conjugate gradient method for solving the system, the iteration is stopped when the relative residual is below $10^{-10}$.
\revHH{The algorithm is implemented using Matlab and run by a single process machine with Intel(R) Core(TM) i7-3520M CPU 2.90GHz and 16GB memory.}

We have tested and compared the following five methods:
\begin{itemize}

\item method0 : standard BDDC with primal unknowns at corners and multiplicity scaling

\item method1 : the two types of generalized eigenvalue problems in \eqref{KimChung:EIG:1} and \eqref{KimChung:EIG:2} with multiplicity scaling


\item method2 : the generalized eigenvalue problem in \eqref{Klawonn:EIG} with multiplicity scaling


\item methods 3 and 4 : the generalized eigenvalue problem in \eqref{Klawonn:EIG} with deluxe scaling

\end{itemize}
We note that in methods 0 to 3, BDDC algorithm with a change of basis is used,
and in method4, the FETI-DP algorithm with
a projector preconditioner is considered.
For all the methods, the condition numbers are controlled by $C \lambda_{TOL}$.
Depending on the choice of scaling matrices the set of adaptive primal constraints can be different.
In our numerical experiments, we will see that method3 with the deluxe scaling gives the smallest set of adaptive constraints among the
three methods, i.e., methods 1 to 3. In method3 and method4, the set of adaptive set of primal constraints is the same while the constraints are enforced by a projection in method4.
\revHH{We note that in \cite{KRR-2015} method4 is tested for various examples in two dimensions.
In that approach, change of basis formulation is not considered but economic version of
method4 is developed and tested to reduce the computational cost. As we will see, the change of
basis formulation seems to be computationally more efficient and stable than the projector preconditioning since the calculation of projection adds considerable cost and instability as the problem size increases, see also numerical results for 3D examples.}

We consider model problems with $\rho(x)$ having some high contrast channel patterns as shown
in Figure~\ref{Fig2:channel}. For this example, we perform the four methods with the same
$\lambda_{TOL}=1+\log(H/h)$.
The results for the five methods are presented in Table~\ref{TB3:channel}.
We note that since $\rho(x)$ is symmetric
across the subdomain interfaces, the set with the first type of primal unknowns in method1 is empty.
In addition, the four methods give the same set of adaptive primal constraints and the same bound
of condition numbers. The minimum eigenvalues are all one and the maximum eigenvalues
are the same for all the four methods.
\revHH{For the case when the adaptive constraints are not employed, we observe big condition numbers
and more iterations.}
\begin{table}
\footnotesize \caption{Performance of the methods 0 to 4 with the same $\lambda_{TOL}=1+\log(H/h)$ for $\rho(x)$ with channels ($p=10^6$): $N_d=3^2$, Iter (number of
iterations), $\lambda_{\min}$ (minimum eigenvalues), $\lambda_{\max}$ (maximum eigenvalues),
pnum1 (number of first type of primal unknowns in method1),
pnum2 (number of second type of primal unknowns in method1, or number of primal unknowns
from GEIG problem in methods 2 to 4), and time (time spent on PCG solver and
eigenvalue problems). }
\label{TB3:channel} \centering
\begin{tabular}{|c|c||c|c|c|c|c|c|c|} \hline
Channels & $H/h$ & method & pnum1 & pnum2 & iter & $\lambda_{\min}$ & $\lambda_{\max}$ & time \\
\hline
three & 14 & method0 &   &     & 15  & 1.00 & 126.26 & 0.37 \\
      &    & method1 & 0 & 24  & 5  & 1.00 & 1.03 & 0.32 \\
      &    & method2 &   & 24  & 5  & 1.00 & 1.03 & 0.27 \\
      &    & method3 &   & 24  & 5  & 1.00 & 1.03 & 0.21  \\
      &    & method4 &   & 24  & 5  & 1.00 & 1.03 & 0.15  \\
\hline
      & 28 & method0 &   &     & 17  & 1.00 & 159.11 & 0.99 \\
      &    & method1 & 0 & 24 & 6 & 1.00& 1.15 & 1.12 \\
      &    & method2 &   & 24 & 6 & 1.00& 1.15 & 1.51 \\
      &    & method3 &   & 24 & 6 & 1.00& 1.15 & 0.71 \\
      &    & method4 &   & 24 & 6 & 1.00& 1.15 & 0.34 \\
\hline
      & 42 & method0 &   &     & 19  & 1.00 & 178.21 & 2.40 \\
      &    & method1 & 0 & 24 & 7 & 1.00 & 1.24 & 2.30 \\
      &    & method2 &   & 24 & 7 & 1.00 & 1.24 & 3.22 \\
      &    & method3 &   & 24 & 7 & 1.00 & 1.24 & 3.30   \\
      &    & method4 &   & 24 & 7 & 1.00 & 1.24 & 0.68   \\
\hline
\end{tabular}
\end{table}

\begin{figure}[t]
\centering
\includegraphics[width=5cm,height=5cm]{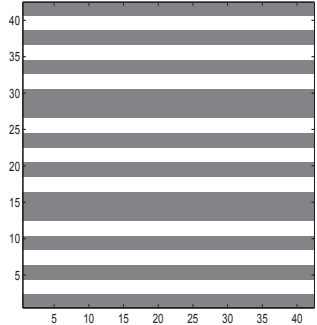}
\caption{In a $3 \times 3$ subdomain partition with three channels in each subdomain with
$H/h=14$: grey ($\rho(x)=1$) and white ($\rho(x)=p$).}
\label{Fig2:channel}
\end{figure}


\begin{figure}[t]
\centering
\includegraphics[height=5cm]{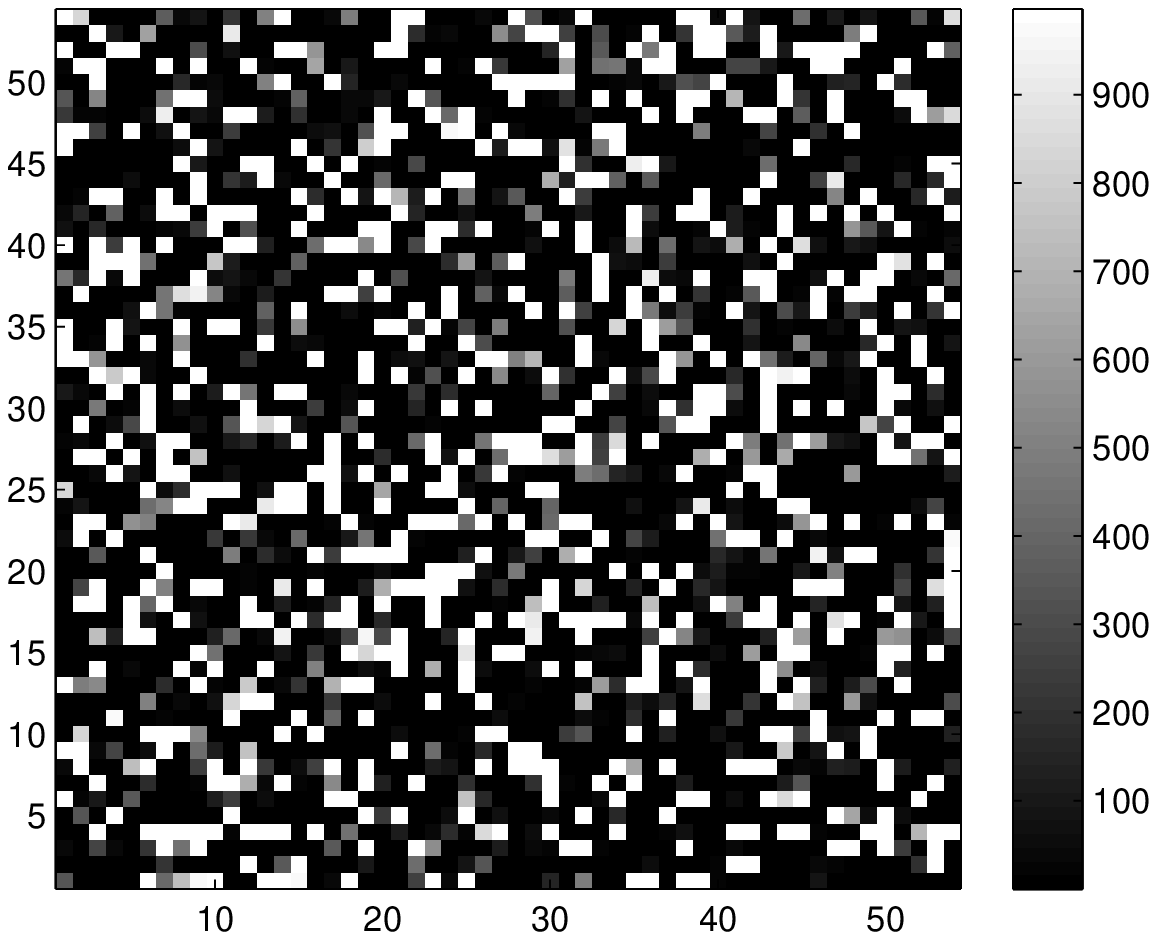}
\caption{$\rho(x)$ for the given $H/h=18$ and $N_d=3 \times 3$.}
\label{Fig1:rand18}
\end{figure}

To test the five methods for highly varying and random coefficients,
we consider $\rho(x)=10^{r}$ where $r$ is chosen
randomly from $(-3,3)$ for each fine grid element.
For the given $N_d=3 \times 3$, we perform
our algorithm for increasing $H/h$. In Figure~\ref{Fig1:rand18}, the value of $\rho(x)$
is plotted for $H/h=18$ and $N_d=3\times 3$.
The results are presented in Table~\ref{TB4:random}. For methods 1 and 2,
the number of adaptive primal unknowns is a considerable size, for an example, about 70\% of total interface unknowns for method1
and 60\% for method2 when $H/h=24$, and they seem to be not a feasible approach
for these test models. On the other hand, methods 3 and 4 with the deluxe scalings give
only about 20 adaptive constraints for all the test models, which clearly shows that
they produce a coarse problem quite robust to highly random coefficients.
\revHH{The results in method0 show that adaptive primal constraints are important in
obtaining efficient BDDC/FETI-DP methods.}
\begin{table}[thp!]
\footnotesize \caption{Performance of methods 0 to 4 for
the problem with random $\rho(x)$ in $(10^{-3},10^3)$ by increasing $H/h$ in a fixed subdomain partition $N_d=3^2$ and $\lambda_{TOL}=1+\log(H/h)$:
Iter (number of iterations), $\kappa$ (condition numbers), time (time spent on PCG solver and
eigenvalue problems),
pnum1 (number of first type of primal unknowns in method1), and
pnum2 (number of second type of primal unknowns in method1, or number of primal unknowns
from the GEIG problem in methods 2 to 4).
}
\label{TB4:random} \centering
\begin{tabular}{|c||c|c|c|c|c|c|c|} \hline
 $H/h$ & method &  Iter & $\kappa$ & time &  pnum1 & pnum2
\rule{0pt}{9pt}
\\  \hline
6   &  method0 &  111  & 5.221e+3    & 0.34 &  &   \\
    &  method1 &  5  & 1.23    & 0.08 & 42 &  14 \\
    &  method2 &  10 & 2.53    & 0.10 &    &  50  \\
    &  method3 &  7  & 1.30    & 0.07 &    &  17 \\
    &  method4 &  7  & 1.30    & 0.06 &    &  17 \\
\hline
12  &  method0 &  217 & 1.575e+4    & 0.84 &  &   \\
    &  method1 &  12 & 2.67    & 0.26 & 84 &  20 \\
    &  method2 &  17 & 3.29    & 0.24 &    &  89 \\
    &  method3 &  9  & 1.68    & 0.20 &    &  23 \\
    &  method4 &  9  & 1.68    & 0.13 &    &  23 \\
\hline
18  &  method0 &  271 & 1.866e+4        & 1.69 &  &  \\
    &  method1 &  11 & 1.81        & 0.42 & 139 &  23 \\
    &  method2 &  18 & 3.62        & 0.40 &     &  129\\
    &  method3 &   9 & 1.81        & 0.34 &     &  21\\
    &  method4 &   9 & 1.81        & 0.21 &     &  21\\
\hline
24  &  method0 &  351      & 1.860e+4      & 2.39  &  &  \\
    &  method1 &  11      & 1.64       & 0.66  &186  &  23 \\
    &  method2 &  19      & 4.03       & 0.65  &     &  172\\
    &  method3 &  11      & 1.96       & 0.58  &     &  20 \\
    &  method4 &  11      & 1.96       & 0.41  &     &  20 \\
\hline
30  &  method0 &  532      & 3.302e+4       & 4.72  &   &  \\
    &  method1 &  14      & 3.61       & 1.07  & 242  &  24 \\
    &  method2 &  21      & 4.31       & 0.97  &      &  201  \\
    &  method3 &  10      & 2.63       & 1.30  &      &  20 \\
    &  method4 &  10      & 2.63       & 0.42  &      &  20 \\
\hline
\end{tabular}
\end{table}

In Table~\ref{TB:coarse:Nd}, the five methods are
tested for highly varying and random $\rho(x)$
by increasing $N_d=N^2$ with a fixed local problem size $H/h=16$.
We observe similar performance to the previous case.
For the methods 1 and 2, the number of adaptive constraints becomes
problematic as $N$ increases, about 12 constraints per edge in method1
and about 9 constraints in method2 with the total unknowns $H/h=16$ per edge.
In methods 3 and 4, about one or two adaptive constraints are chosen
per edge. Again methods 3 and 4 can provide a scalable and robust coarse problem
for the test models even increasing $N$ with highly random coefficients.
On the other hand, when increasing $N$, the timing result in method4 shows that the cost for projection
becomes problematic and an efficient way of implementing the projection needs to be investigated.

\begin{table}[thp!]
\footnotesize \caption{Performance of the methods 0 to 4 with the same $\lambda_{TOL}=1+\log(H/h)$ for
highly varying and random $\rho(x)$ in $(10^{-3},10^3)$ by increasing $N_d=N^2$
and a fixed $H/h=16$:
$N$ (number of subdomains in one direction),
Iter (number of iterations), $\kappa$ (condition numbers), time (time spent
on PCG solver and eigenvalue problems),
pnum1 (number of first type of primal unknowns in method1),
and pnum2 (number of second type of primal unknowns in method1 or number of primal
unknowns from GEIG problem in methods 2 to 4).
}
\label{TB:coarse:Nd} \centering
\begin{tabular}{|c||c|c|c|c|c|c|} \hline
$N$ & method &  Iter & $\kappa$ & time & pnum1 & pnum2
\rule{0pt}{9pt}
\\  \hline
4   & method0   & 423 & 1.019e+4 & 3.85 &  & \\
    & method1   & 12 & 1.96 & 0.59 & 243 & 45\\
    & method2   & 19 & 3.74 & 0.65 &  & 235\\
    & method3   & 11 & 1.74      & 0.54 &  & 42\\
    & method4   & 11 & 1.74      & 0.46 &  & 42\\
   \hline
8   & method0   & $>$1000 & $>$3.567e+4      & 36.87  &  &  \\
    & method1   & 17 & 4.10       & 4.11  & 1165 & 186 \\
    & method2   & 19 & 3.72 & 3.45 &  & 1064\\
    & method3   & 16  & 3.11      & 2.56  &  & 189\\
    & method4   & 16  & 3.11      & 5.94  &  & 189\\
   \hline
16  & method0   & $>$1000 & $>$5.208e+4  & 150.55 &  &  \\
    & method1   & 20 & 4.96  & 101.20 & 5024 & 756 \\
    & method2  & 19 & 3.73 & 63.41 &  & 4584\\
    & method3   & 17   & 2.69 & 10.67 & & 805\\
    & method4   & 17   & 2.69 & 88.65 & & 805\\
   \hline
\end{tabular}
\end{table}

In Table~\ref{TB8:fracture},
performance of the five methods is presented for a fracture-like medium, shown in Figure~\ref{fracture},
by varying the contrast value $p$. For the test models, the number of adaptive constraints per edge
is about one or two in the three methods. Similarly to the previous example, method3 shows the best performance
with the smallest set of adaptive primal constraints.
We note that the condition numbers and iteration counts in Tables~\ref{TB3:channel}-\ref{TB8:fracture}
confirm our theoretical estimate in Section~\ref{analysis}.

\begin{figure}[t]
\centering
\includegraphics[width=6cm,height=6cm]{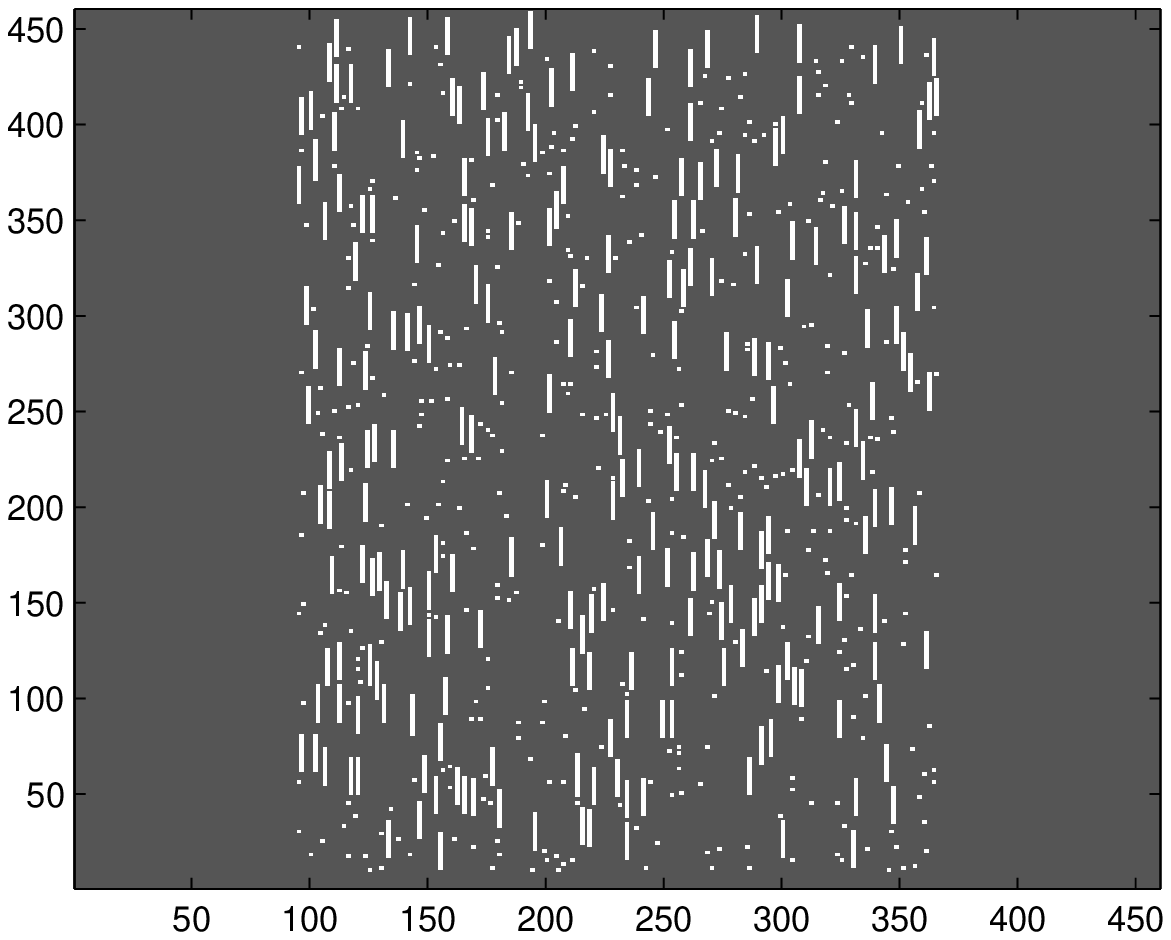}
\caption{A fracture-like medium: grey represents value $1$ and white represents the value $p$,
which is chosen as $1,10^3$ and $10^6$ in our numerical simulations.}
\label{fracture}
\end{figure}

\begin{table}[thp!]
\footnotesize \caption{Performance of methods 0 to 4 for
a fracture-like medium with varying $p$, $\lambda_{TOL}=1+\log(H/h)$, $H/h=23$, and $N=20$:
Iter (number of iterations), $\kappa$ (condition numbers),
time (time spent on PCG solver and eigenvalue problems),
pnum1 (number of first type of primal unknowns in method1),
and pnum2 (number of second type of primal unknowns in method1, or
number of primal unknowns from GEIG problem in methods 2 to 4).
}\label{TB8:fracture} \centering
\begin{tabular}{|c||c|c|c|c|c|c|} \hline
$p$& method&  Iter & $\kappa$ & time & pnum1 & pnum2
\rule{0pt}{9pt}
\\  \hline
1   & method0 & 19  & 4.00 & 25.53 &  &  \\
    & method1 & 9  & 1.46 & 26.02 & 0 & 760 \\
    & method2 & 9  & 1.46 & 26.10 &   & 760 \\
    & method3 & 9  & 1.46 & 25.52 &  & 760 \\
    & method4 & 9  & 1.46 & 203.35 &  & 760 \\
   \hline
$10^3$  & method0 &  183 & 488.88            & 78.19 &  & \\
        & method1 &  15 & 2.87            & 107.76 & 539 & 818 \\
        & method2 & 17   &  3.45  & 117.19    &   & 1345 \\
        & method3 & 13   &  2.07     & 45.66 &   & 794 \\
        & method4 & 13   &  2.07     & 215.18 &   & 794 \\
   \hline
$10^6$  & method0 &  $>$556 & N/A & 200.965 &  &   \\
        & method1 &  15 & 2.96 & 59.67 & 541 & 818  \\
        & method2 & 17       & 3.88 & 45.11 &   & 1349 \\
        & method3 & 13       & 2.09         & 26.77 &   & 795 \\
         & method4 & 13       & 2.09         & 214.00 &   & 795 \\
   \hline
\end{tabular}
\end{table}

\subsection{$3D$ case}\label{sec:3d}
In the three-dimensional case, we consider $\Omega$ to be
the unit cube $(0,1)^3$ and decompose the domain into $N^3$ subdomains with side length $H=1/N$. Each subdomain is then divided into uniform tetrahedra with size $h$.
We assume again that the meshes in different subdomains are matching on common faces and edges.
The CG iteration is stopped when the relative residual has been reduced by $10^{-10}$. We consider
the same four methods, methods 1 to 4, as in the previous subsection for faces
with the additional generalized eigenvalue problems \eqref{eq:Geig:edge} on edges as introduced in Section~\ref{sec:GEIG:edge} and perform the four methods with given values $\lambda^F_{TOL}=1+\log(H/h)$ and $\lambda^E_{TOL}=4H/h$ for face and edge GEIG problems, respectively. The estimate
of the condition numbers is then $C \lambda_{TOL}$ for all the four methods.
\revHH{The algorithm is implemented using Matlab and run by a single process machine with Intel(R)~Xeon(R)~CPU~X5650~2.67GHz and 64GB memory.}

We first consider model problems with $\rho(x)$ having some high contrast channel patterns as shown in  Figure~\ref{Fig2:channel-3d}.
In Table~\ref{3d-channel-eta}, we list the results of the four methods for a fixed subdomain partition $N_d=3^3$ and a fixed $H/h=12$ by varying the contrast $p$. We can see that method4 is not stable when $p\ge 10^3$, but methods 1 to 3 work well even if $p$ is very large.
The numerical instability in method4 is caused by the ill conditioning and roundoff error in the calculation of the projection $P$ as discussed in \cite{PP-FETI-DP}. The four methods give the same set of primal unknowns, the minimum eigenvalue as one, and the same maximum eigenvalues. The performance is similar to the previous case that method1 is the most efficient in view of timing results.

\begin{figure}[t]
\centering
\includegraphics[width=6cm,height=6cm]{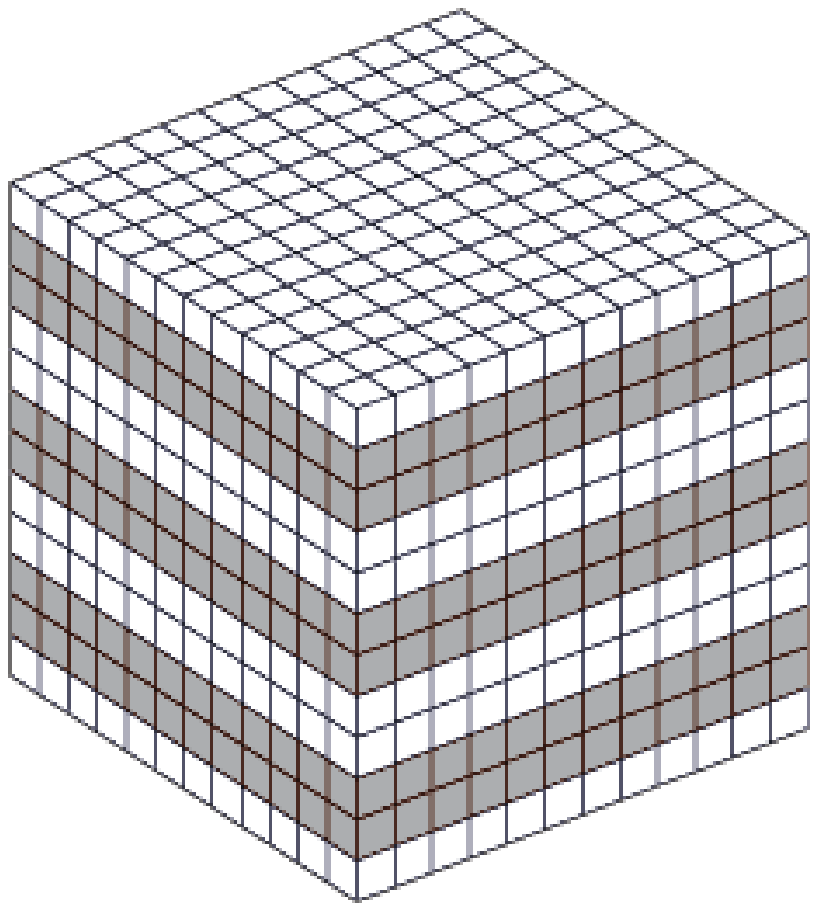}
\caption{In a $3 \times 3\times3$ subdomain partition, $\rho(x)$ with one channel for $H/h=4$: white ($\rho(x)=1$) and grey ($\rho(x)=p$).}
\label{Fig2:channel-3d}
\end{figure}

\begin{table}
\footnotesize \caption{Performance of the methods 1 to 4 with the same $\lambda^F_{TOL}=1+\log(H/h),\lambda^E_{TOL}=4H/h$ for $\rho(x)$ with one channel ($p=10,10^2,10^3,10^4,10^6$): $N_d=3^3$ and $H/h=12$, Iter (number of
iterations), $\lambda_{\min}$ (minimum eigenvalues), $\lambda_{\max}$ (maximum eigenvalues), time (time spent
on PCG solver and eigenvalue problems),
pnum1 (number of first type of primal unknowns on faces in method1), and
pnum2 (number of second type of primal unknowns on faces in method1, or number of primal unknowns
from the GEIG problem on faces in methods 2 to 4), pnumE (number of primal unknowns on edges). $p1=\frac{pnum1}{M_F}$, $p2=\frac{pnum2}{M_F}$, $pE=\frac{pnumE}{M_E}$, $M_F, M_E$ are the numbers of faces and edges, respectively. }
\label{3d-channel-eta} \centering
\begin{tabular}{|c||c|c|c|c|c|c|c|c|c|c|c|} \hline
$p$ & method & pnum1 & pnum2 & pnumE & iter & $\lambda_{\min}$ & $\lambda_{\max}$ & time & $p1$& $p2$& $pE$ \\
\hline
$10$      & method1 & 0 & 54 &36 & 11& 1.00 & 1.64 & 79.81  &0&1&1 \\
          & method2 &   & 54 &36 & 11& 1.00 & 1.64 & 88.48  & &1&1 \\
          & method3 &   & 54 &36 & 11& 1.00 & 1.64 & 1.17e+3& &1&1 \\
          & method4 &   & 54 &36 & 13& 1.01 & 1.96 & 2.36e+3& &1&1 \\
\hline
$10^2$ & method1 & 0 & 54 &36 & 10& 1.00 & 1.63 & 80.79  &0&1&1 \\
       & method2 &   & 54 &36 & 10& 1.00 & 1.63 & 93.21  & &1&1 \\
       & method3 &   & 54 &36 & 10& 1.00 & 1.62 & 1.17e+3& &1&1     \\
       & method4 &   & 54 &36 & 12& 1.00 & 1.95 & 1.86e+3& &1&1     \\
\hline
  $10^3$  & method1 & 0 & 54 &36 & 10    & 1.00 & 1.62 & 78.89  &0&1&1 \\
          & method2 &   & 54 &36 & 10    & 1.00 & 1.62 & 93.99  & &1&1 \\
          & method3 &   & 54 &36 & 10    & 1.00 & 1.62 & 1.17e+3& &1&1     \\
          & method4 &   & 54 &36 &$>1000$& & & & & &      \\
\hline
$10^4$ & method1 & 0 & 54 &36 & 10    & 1.00 & 1.62 & 79.84  &0&1&1 \\
       & method2 &   & 54 &36 & 10    & 1.00 & 1.62 & 96.12  & &1&1 \\
       & method3 &   & 54 &36 & 10    & 1.00 & 1.62 & 1.17e+3& &1&1     \\
       & method4 &   & 54 &36 &$>1000$& & & & & &      \\
\hline
$10^6$ & method1 & 0 & 54 &36 & 10    & 1.00 & 1.62 & 79.66  &0&1&1 \\
       & method2 &   & 54 &36 & 10    & 1.00 & 1.62 & 117.50  & &1&1 \\
       & method3 &   & 54 &36 & 10    & 1.00 & 1.62 & 1.18e+3& &1&1     \\
       & method4 &   & 54 &36 &$>1000$& & & & & &      \\
\hline
\end{tabular}
\end{table}

\revHH{For the above channel model, we can consider an economic version for method3 to reduce the computational cost. As we can see in Table~\ref{3d-channel-eta-time}, most computing time was spent in forming generalized eigenvalue problems on each face and edge. In the economic version, see \cite{KRR-2015}, the matrices considered in
the generalized eigenvalue problems are replaced with those obtained from the local stiffness matrices
restricted to the slab of faces and edges.
The thickness $\eta$ of the slab is chosen to be $h$ for each face and edge. When $\eta$ is chosen
to be $H$, the width of the subdomain, the economic version is identical to the original one.
The performance of the economic version is compared to the original one in Table~\ref{3d-channel-eta-full}.
Thus by using the slab with $\eta=h$ the computing cost is greatly reduced but a larger set of primal constraints is obtained. For both approaches, the results are robust to the contrast of the channel.
In Figure~\ref{Fig-face-eig-312-channel-p1000}-\ref{Fig-edge-eig-312-channel-p1000}, we plot eigenvalues of each face and edge with $\eta=h$
and $\eta=H$ to study effective choices for $\lambda^F_{TOL}$ and $\lambda^E_{TOL}$.
The eigenvalues are plotted except that corresponding to infinity.
For other values of $p$,
the patterns of eigenvalues are similar. For the e-version, the eigenvalues are larger and thus larger values of
$\lambda_{TOL}$ could give a smaller set of primal constraints. For the above channel models, even with larger
values of $\lambda^E_{TOL}$ good condition numbers are still obtained.
}


\begin{table}
\footnotesize \caption{Comparison of timing results in method3 with e-version(see \cite{KRR-2015}, $\eta=h$)
and without e-version for $\rho(x)$ with one channel ($p=10,10^2,10^3$ and $H/h=12$, $N_d=3^3$): GEP-F (computing time in forming generalized eigenvalue problem for face), GEP-E (computing time
in forming generalized eigenvalue problem for edge), Deluxe-FE (computing time in forming deluxe scalings $D_F$ and $D_E$). }
\label{3d-channel-eta-time} \centering
\begin{tabular}{|c||c|c|c|c|} \hline
$p$ & e-version &forming-GEIG-F & forming-GEIG-E & Deluxe-FE \\
\hline
$10$   & with   &274.75 &32.46   & 0.07 \\
       & without&510.04 &681.53  & 0.11 \\
\hline
$10^2$ & with   &297.29&33.42 & 0.07  \\
       & without&504.36&676.88& 0.18  \\
\hline
$10^3$ & with   &260.66&29.16  & 0.06  \\
       & without&497.12&684.79 & 0.15  \\
\hline
\end{tabular}
\end{table}

%

\begin{figure}[t]
\centering
\includegraphics[width=14cm,height=8cm]{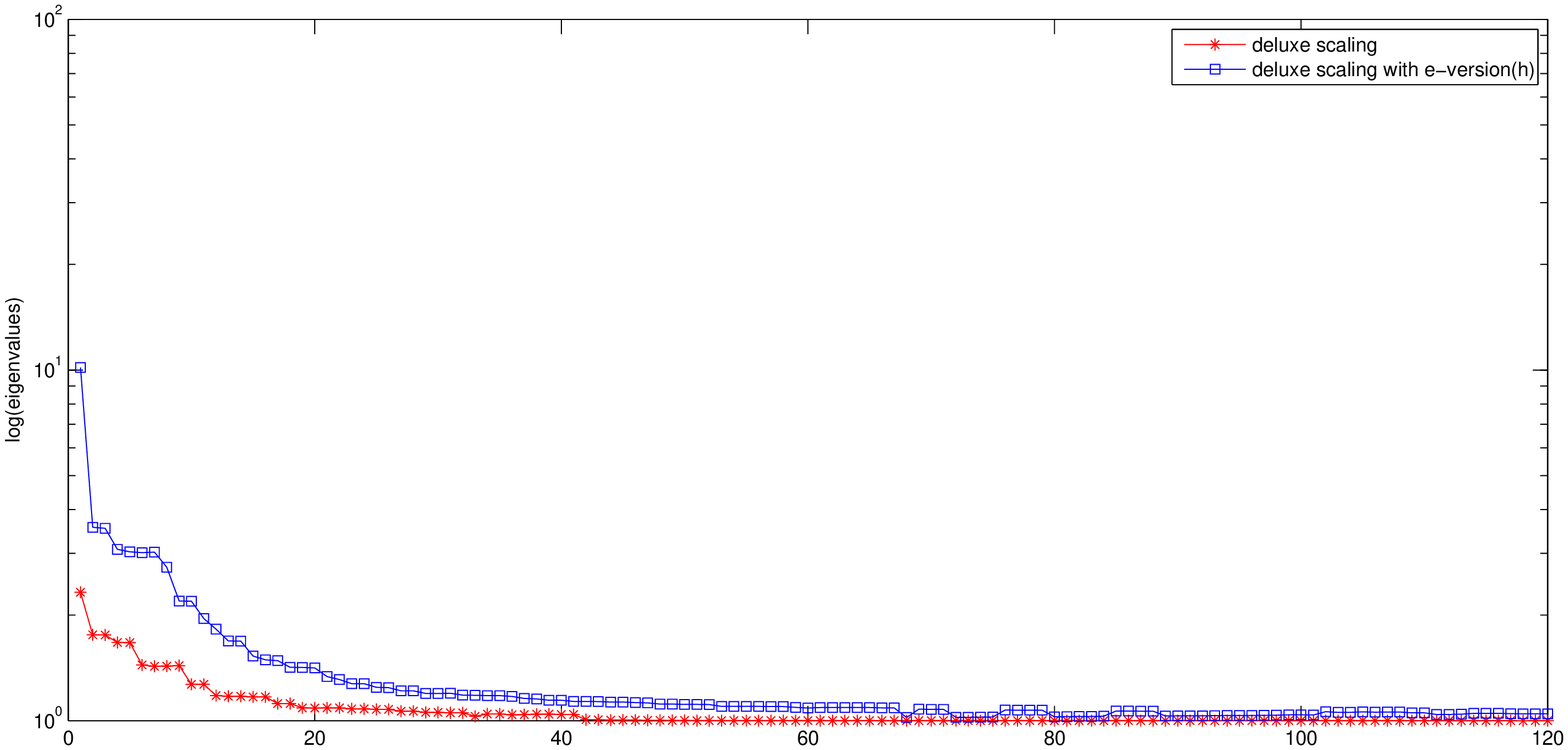}
\caption{Plot of eigenvalues except infinity for a face: channel model with $p=1000$ and $H/h=12$.}
\label{Fig-face-eig-312-channel-p1000}
\end{figure}



\begin{figure}[t]
\centering
\includegraphics[width=14cm,height=8cm]{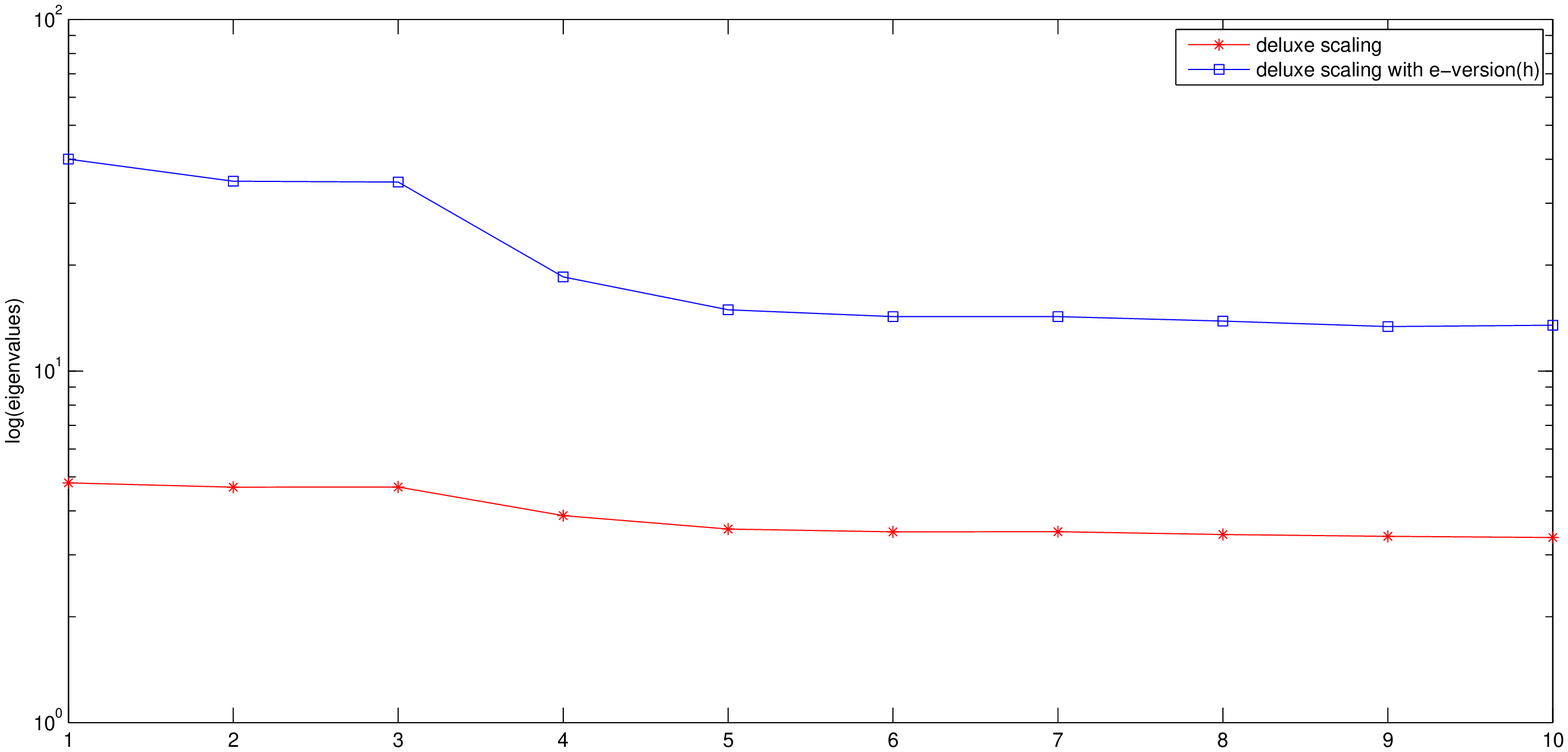}
\caption{Plot of eigenvalues except infinity for an edge: channel model with $p=1000$ and $H/h=12$.}
\label{Fig-edge-eig-312-channel-p1000}
\end{figure}

\begin{table}
\footnotesize \caption{Comparison of method3 with e-version ($\eta=h$, $\lambda^E_{TOL}=1000$) and without e-version ($\eta=H$, $\lambda^E_{TOL}=4H/h$) and $\lambda^F_{TOL}=1+\log(H/h)$ for $\rho(x)$ with one channel ($p=10,10^2,10^3$) in each subdomain and with $N_d=3^3$ and $H/h=12$: Iter (number of
iterations), $\lambda_{\min}$ (minimum eigenvalues), $\lambda_{\max}$ (maximum eigenvalues), time (time spent
on PCG solver and eigenvalue problems),
pnumF (number of primal unknowns on faces), pnumE (number of primal unknowns on edges). $pF=\frac{pnumF}{M_F}$, $pE=\frac{pnumE}{M_E}$, $M_F, M_E$ are the numbers of faces and edges, respectively.}
\label{3d-channel-eta-full} \centering
\begin{tabular}{|c||c|c|c|c|c|c|c|c|c|} \hline
$p$ & $\eta$ & pnumF & pnumE & iter & $\lambda_{\min}$ & $\lambda_{\max}$ & time & $pF$ & $pE$ \\
\hline
$10$   & $h$ & 252   &36 & 9  & 1.00 & 1.50  & 315.41  &4.67 &1 \\
       & $H$ & 54    &36 & 11 & 1.00 & 1.64  & 1.17e+3 & 1   &1 \\
       \hline
$10^2$ & $h$ & 216   &36 & 9  & 1.00 & 1.61  & 338.32  &4.00 &1 \\
       & $H$ & 54    &36 & 10 & 1.00 & 1.62  & 1.17e+3 &1    &1 \\
\hline
$10^3$ & $h$ & 216   &36 & 9  & 1.00 & 1.62  & 296.31  &4.00 &1 \\
       & $H$ &  54   &36 & 10 & 1.00 & 1.62  & 1.17e+3 &1    &1 \\
\hline
\end{tabular}
\end{table}

We now consider highly varying and random coefficients $\rho(x)=10^{r}$ where $r$ is chosen randomly from $(-3,3)$ for each fine hexahedral grid element.
As an example, the value of $\rho(x)$ is presented for $H/h=2$ and $N_d=2^3$ in Figure~\ref{Fig1:rand18-3d}.
For a given $N_d=3^3$, we perform
our algorithm for increasing $H/h$ in  Table~\ref{3d-random-Hh}. We observe that for method1 and method2, the number of adaptive primal unknowns is still a considerable size as in two dimensions, for example, about 67\% of total face interior unknowns for method1 and 45\% for method2, 81\% of total edge interior unknowns for both methods 1 and 2 when $H/h=12$. On the other hand, method3 with deluxe scalings gives only less than 4 adaptive constraints on each face, but still as much edge constraints as method1 and method2. The iteration counts of methods 1 and 2 are almost the same. For method3, the iteration counts is less than methods 1 and 2. 
In addition, method2 becomes more and more efficient than method1 considering the timing results when $H/h$ grows. This is due to the fact that the cost for additional generalized eigenvalue problems in
method1 exceeds the cost for the parallel sum in method2 for this test example.

\begin{figure}[t]
\centering
\includegraphics[width=16cm,height=6cm]{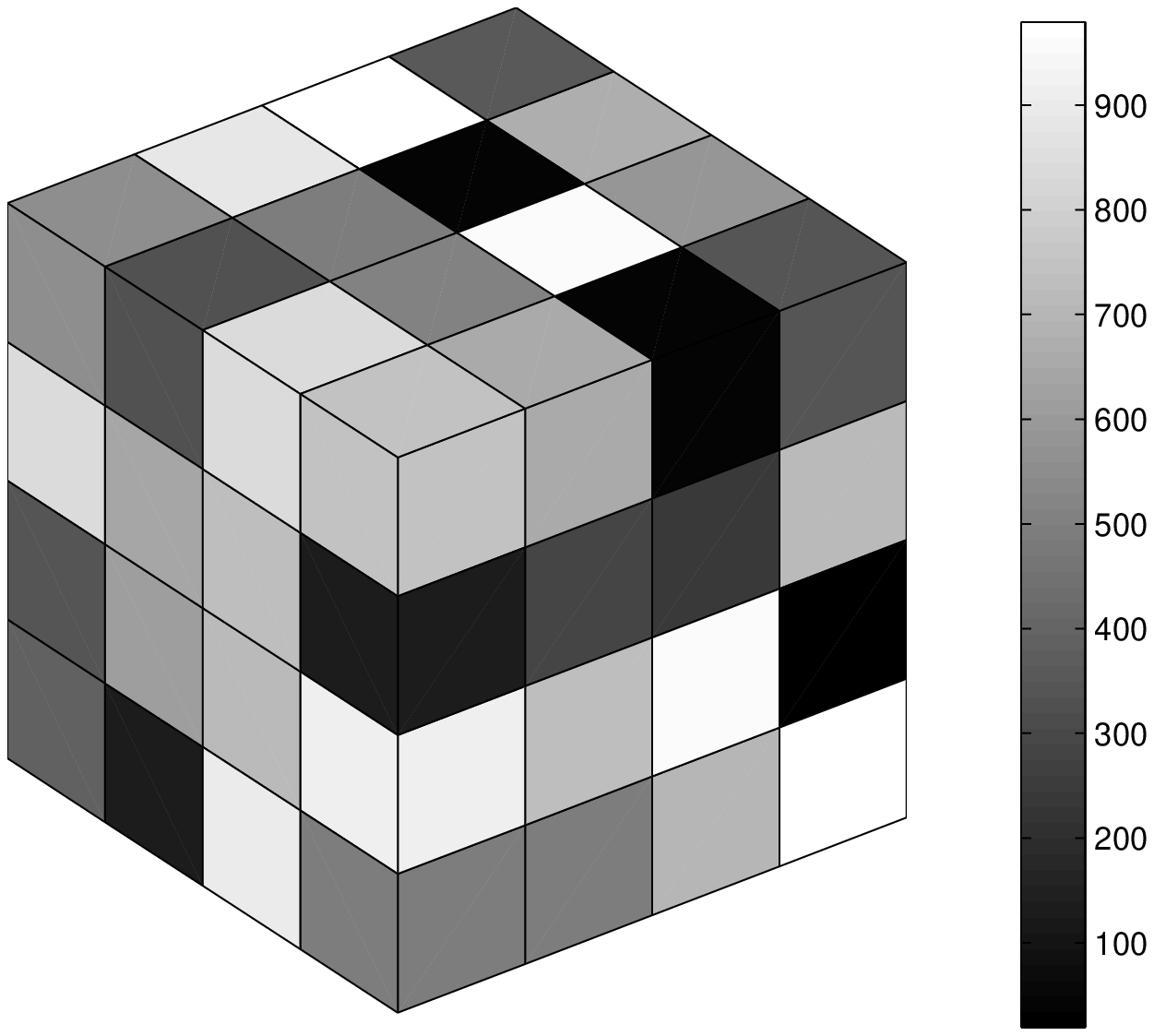}
\caption{ $\rho(x)$ for the given $H/h=2$ and $N_d=2 \times 2 \times 2$.}
\label{Fig1:rand18-3d}
\end{figure}


\begin{table}[thp!]
\footnotesize \caption{Performance of methods 1 to 4 for
the problem with random $\rho(x)$ in $(10^{-3},10^3)$ by increasing $H/h$ in a fixed subdomain partition $N_d=3^3$ and $\lambda^F_{TOL}=1+\log(H/h), \lambda^E_{TOL}=4H/h$:
Iter (number of iterations), $\kappa$ (condition numbers), time (time spent on PCG solver and
eigenvalue problems),
pnum1 (number of first type of primal unknowns on faces in method1), and
pnum2 (number of second type of primal unknowns on faces in method1, or number of primal unknowns
from the GEIG problem on faces in methods 2 to 4), pnumE (number of primal unknowns on edges). $p1=\frac{pnum1}{M_F}$, $p2=\frac{pnum2}{M_F}$, $pE=\frac{pnumE}{M_E}$, $M_F, M_E$ are the numbers of faces and edges, respectively.
}
\label{3d-random-Hh} \centering
\begin{tabular}{|c||c|c|c|c|c|c|c|c|c|c|c|} \hline
 $H/h$ & method &  Iter & $\kappa$ & time &  pnum1 & pnum2 & pnumE & $p1$& $p2$& $pE$
\rule{0pt}{9pt}
\\  \hline
4   &  method1 &  13 & 3.97    & 1.67  & 336 &  70 &105 &6.22&1.30&2.92\\
    &  method2 &  15 & 3.98    & 2.97  &     & 319 &105 & &5.91&2.92\\
    &  method3 &  10 & 1.47    & 13.26 &     &  91 &103 & &1.69&2.86\\
    &  method4 &  15 & 4.62    & 11.37 &     &  91 &103 & &1.69&2.86\\
\hline
8   &  method1 &  25 & 7.70    & 27.35 & 1771 &  134 & 218&32.80&2.50&6.10\\
    &  method2 &  26 & 7.83    & 29.16 &      &  1292& 218&&23.90&6.10\\
    &  method3 &  12 & 1.89    & 181.86&      &  147 & 201&&2.72 &5.58\\
    &  method4 &  18 & 3.92    & 235.73&      &  147 & 201&&2.72 &5.58\\
\hline
12  &  method1 &  30 & 12.03        & 203.99  & 4148 &  198 &320&76.80&3.70&8.90\\
    &  method2 &  31 & 12.21        & 132.01  &      &  2949&320&&54.60&8.90\\
    &  method3 &  15 & 2.41         & 1.20e+3 &      &  190 &289&&3.52&8.03\\
    &  method4 &  26 & 12.88        & 3.13e+3 &      &  190 &289&&3.52&8.03\\
\hline
16  &  method1 &  40 & 18.15        & 917.99  & 7633 &  241 &394&141.40&4.50&10.90\\
    &  method2 &  42 & 18.15        & 494.63  &      &  5018&394&&92.90&10.90\\
    &  method3 &  17 & 3.65         & 5.11e+3 &      &  237 &336& &4.39&9.33  \\
    &  method4 &  26 & 7.04         & 6.05e+3 &      &  237 &336& &4.39&9.33 \\
\hline
\end{tabular}
\end{table}

\begin{figure}[t]
\centering
\includegraphics[width=14cm,height=8cm]{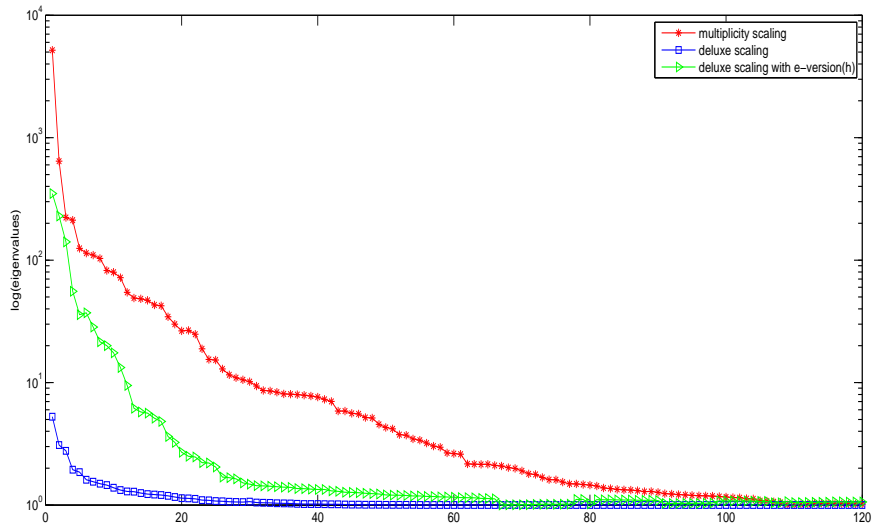}
\caption{Plot of eigenvalues for a face except infinity: random $\rho(x)$, $H/h=12$.}
\label{Fig312-random-face-eig-ms-ds-e-version}
\end{figure}
\begin{figure}[t]
\centering
\includegraphics[width=14cm,height=8cm]{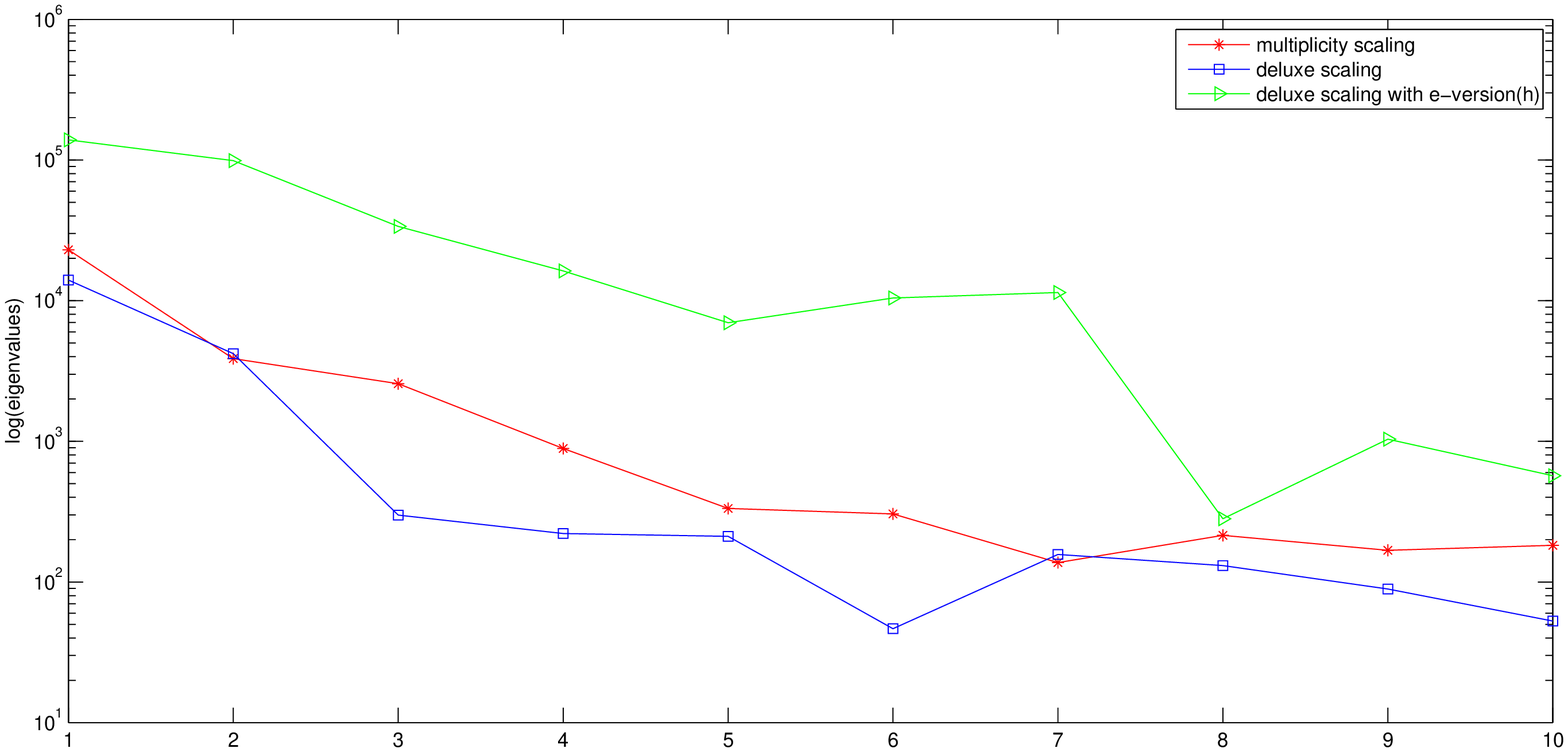}
\caption{Plot of eigenvalues for an edge except infinity: random $\rho(x)$, $H/h=12$.}
\label{Fig312-random-edge-eig-ms-ds-e-version}
\end{figure}

\revHH{To study the effective choice for $\lambda_{TOL}$, we plot the eigenvalues of the generalized eigenvalue problems for a face $F$ and an edge $E$ in the model considered in Table~\ref{3d-random-Hh}.
We plot eigenvalues for multiplicity scalings, deluxe scalings with and without e-version; see Figures~\ref{Fig312-random-face-eig-ms-ds-e-version} and \ref{Fig312-random-edge-eig-ms-ds-e-version}.
We can see that edge eigenvalues are much larger than those for faces and  we thus choose $\lambda_E=1000$ and compare the performance of the four methods in Table~\ref{3d-random-Hh-lE1000}.
For methods 3 and 4 with deluxe scalings, the condition numbers and iteration counts seem to be quite robust to the larger choice of $\lambda_E$ and with much smaller set of primal unknowns on edges, about less than half for the case with $\lambda_E=4 H/h$.
On the other hand, for methods 1 and 2 with multiplicity scalings the condition numbers and iteration counts greatly increase.
In addition, to reduce computational cost in method3 we apply e-version of generalized eigenvalue problems and deluxe scalings. The results are listed in Table~\ref{3d-random-Hh-lE1000}. They show that
e-version greatly improves computational efficiency in method3.
On the other hand, for e-version a larger set of adaptive primal constraints is selected. For the e-version with
larger $\lambda_{TOL}$, we can reduce the number of adaptive primal constraints maintaining good condition numbers.
For the test example, the use of lager $\lambda_{TOL}$ does not reduce the computing time but it can reduce
the computing time when a large coarse problem becomes a bottleneck of computation.}

\revHH{
\begin{table}[thp!]
\footnotesize \caption{Performance of methods 1 to 4 for
the problem with random $\rho(x)$ in $(10^{-3},10^3)$ by increasing $H/h$ in a fixed subdomain partition $N_d=3^3$ and $\lambda^F_{TOL}=1+\log(H/h), \lambda^E_{TOL}=1000$, and e-method3(L) (with $\lambda^F_{TOL}=10\log(H/h)$, $\lambda^E_{TOL}=10^4\log(H/h)$):
Iter (number of iterations), $\kappa$ (condition numbers), time (time spent on PCG solver and
eigenvalue problems),
pnum1 (number of first type of primal unknowns on faces in method1), and
pnum2 (number of second type of primal unknowns on faces in method1, or number of primal unknowns
from the GEIG problem on faces in methods 2 to 4), pnumE (number of primal unknowns on edges). $p1=\frac{pnum1}{M_F}$, $p2=\frac{pnum2}{M_F}$, $pE=\frac{pnumE}{M_E}$, $M_F, M_E$ are the numbers of faces and edges, respectively.
}
\label{3d-random-Hh-lE1000} \centering
\begin{tabular}{|c||c|c|c|c|c|c|c|c|c|c|c|} \hline
 $H/h$ & method &  Iter & $\kappa$ & time &  pnum1 & pnum2 & pnumE & $p1$& $p2$& $pE$
\rule{0pt}{9pt}
\\  \hline
4   &  method1 &  74 & 183.15    & 4.55  & 336 &  70 &65 &6.22&1.30&1.81\\
    &  method2 &  76 & 1.83e+2   & 4.85  &     & 319 &65 & &5.91&1.81\\
    &  method3          &  11 & 2.09    & 15.44 &     &  91 &56 &  &1.69 &1.56\\
    &  e-method3 &9 &1.36 &22.28 &     & 177&83   &  & 3.28 &2.31 \\
    & e-method3(L) &12&2.15&23.62& &109&51& &2.02 &1.42\\
    &  method4 &  15 & 4.65    & 14.54 &     &  91 &56 & &1.69&1.56\\\hline
8   &  method1 &  127 & 263.94    & 51.68 & 1771 &  134 & 101&32.80&2.50&2.80\\
    &  method2 &  129 & 263.83    & 33.47 &      &  1292& 101&&23.90&2.80\\
    &  method3 &  15 & 2.54    & 183.59&      &  147 & 84&&2.72 &2.33\\
    &  e-method3 &11&2.06 &94.37 &     &547 &182 &  & 10.13 &5.06    \\
    &  e-method3(L) &16&4.14&97.76& &276&90& &5.11 &2.50\\
    &  method4 &  19 & 3.95    & 260.11&      &  147 & 84&&2.72 &2.33\\\hline
12  &  method1 &  135 & 2.39e+2     & 730.99  & 4148 &  198 &137&76.80&3.70&3.80\\
    &  method2 &  140 & 2.39e+2     & 305.49  &      &  2949&137&&54.60&3.80\\
    &  method3 &  21 & 5.11         & 1.13e+3 &      &  190 &115&&3.52&3.19\\
    &  e-method3  &12&3.16 &289.09 &     & 1068&288 &  &19.80 &8.00 \\
    &  e-method3(L) &20&4.70&300.35& &499&128& &9.24&3.56\\
    &  method4 &  28 & 12.91        & 3.54e+3 &      &  190 &115&&3.52&3.19\\
\hline
16  &  method1 &  154 & 2.90e+2     & 4.53e+3 & 7633 &  241 &196&141.40&4.50&5.40\\
    &  method2 &  158 & 2.90e+2     & 1.12e+3 &      &  5018&196&&92.90&5.40\\
    &  method3 &  22 & 5.13         & 5.05e+3 &      &  237 &148& &4.39&4.11  \\
    &  e-method3  &12&2.80 &798.96 &     & 1742&354&  &32.30&9.8    \\
    &  e-method3(L) &21&4.89&912.92& &722&143& &13.37&3.97\\
    &  method4 &  27 & 7.04         & 6.27e+3 &      &  237 &148& &4.39&4.11 \\
\hline
\end{tabular}
\end{table}
}

The results for highly varying and random coefficients by increasing $N_d=N^3$ and a fixed $H/h=12$ are presented in Table~\ref{3d-random-Nd}. We can see that for methods 1 and 2, the number of adaptive constraints becomes problematic as $N$ increases, about 80 constraints per face in method1 and about 55 constraints in method2 with the total number $H/h=12$, i.e. 121 interior nodes per face.
In methods 3 and 4, about less than 4 adaptive constraints are chosen per face. The edge constraints which has been chosen in the method3 is a little bit less than methods 1 and 2.
\revHH{By using e-version of method3, we can reduce the computational cost with better conditioner numbers.
To reduce the adaptive primal constraints, larger values of $\lambda_{TOL}$ can be used for the e-version and
the results present good condition numbers and iteration counts.}
In method4, we again observe numerical instability and considerable cost for projection as increasing $N$.
Efficient implementation
of projection operator $P$ should be addressed in elsewhere; see also discussions in \cite{Klawonn-TR}.
\revHH{In conclusion, method3 with the economic version can provide a scalable and robust coarse problem again for the test models even increasing $N$ with highly random coefficients.}

\begin{table}[thp!]
\footnotesize \caption{Performance of the methods 1 to 4 with $\lambda^F_{TOL}=1+\log(H/h)$, and  $\lambda^E_{TOL}=1000$, and e-method3(L) (with $\lambda_{TOL}^F=10$, $\lambda_{TOL}^E=10^4$) for
highly varying and random $\rho(x)$ in $(10^{-3},10^3)$ by increasing $N_d=N^3$
and a fixed $H/h=12$:
$N$ (number of subdomains in one direction),
Iter (number of iterations), $\kappa$ (condition numbers), time (time spent
on PCG solver and eigenvalue problems),
pnum1 (number of first type of primal unknowns on faces in method1), and
pnum2 (number of second type of primal unknowns on faces in method1, or number of primal unknowns
from the GEIG problem on faces in methods 2 to 4), pnumE (number of primal unknowns on edges). $p1=\frac{pnum1}{M_F}$, $p2=\frac{pnum2}{M_F}$, $pE=\frac{pnumE}{M_E}$, $M_F, M_E$ are the numbers of faces and edges, respectively.
}
\label{3d-random-Nd} \centering
\begin{tabular}{|c||c|c|c|c|c|c|c|c|c|c|} \hline
$N$ & method &  Iter & $\kappa$ & time & pnum1 & pnum2 & pnumE & $p1$& $p2$& $pE$
\rule{0pt}{9pt}
\\  \hline
2   & method1   & 65 & 173.26  & 20.56 & 949 & 44 &20 &79.08&3.67&3.33\\
    & method2   & 69 & 178.95  & 18.58 &     & 640&20 & &53.33&3.33\\
    & method3   & 16 & 4.11    & 1.01e+3&     & 46 &13 &&3.83&2.17\\
    & e-method3 &10 &1.80 &65.36 &     & 242&45 & &20.17 &7.50    \\
    & e-method3(L) &14&3.15&63.67& &157&22& &13.08 &3.67\\
    & method4   & 18 & 4.12   & 324.91&     & 46 &17 &&3.83&2.83\\
   \hline
3   & method1   & 135 & 2.39e+2 & 711.86  & 4148 & 198 &137&76.80&3.70&3.80 \\
    & method2   & 140 & 2.39e+2 & 309.83  &      & 2949&137&&54.60& 3.80\\
    & method3   & 20  & 5.56    & 5.88e+3 &      & 190 &84&&3.52& 2.33\\
    & e-method3 &12 &3.16 &196.82 &   &1068 &288 & &19.80 &8.00    \\
    & e-method3(L) &18&4.62&295.53& &725&175& &13.43&4.90\\
    & method4   & 28 & 12.91 & 3.31e+3 &      & 190 &115&&3.52& 3.19\\
   \hline
4   & method1   & 177 & 2.94e+2 & 1.45e+3 & 11165 & 538  &420&78.00&4.00&4.00\\
    & method2   & 187 & 2.95e+2 & 3.27e+3 &       &7764  &420&&53.90&4.00\\
    & method3   & 24  & 8.60    & 3.40e+3 &       & 533  &225&&3.70&2.08\\
    & e-method3 &13 &2.44 &800.36&     & 2835 &819 & &19.70 &7.60    \\
    & e-method3(L) &16&3.36&802.47& &1876&459& &13.00&4.30\\
    & method4   & $>$ 1000 &       &         &       &   & &&&\\
   \hline
\end{tabular}
\end{table}

\section{Conclusion}\label{sec:con}

In this paper, we develop adaptive coarse spaces for the BDDC and FETI-DP
algorithms for second order elliptic problems discretized by
the standard conforming finite elements.
The coarse components are obtained by solving local generalized eigenvalue problems
for edges (in 2D), and faces and edges (in 3D).
We also consider the use of both multiplicity scalings and deluxe scalings,
as well as the change of basis formulation and the projection formulation.
\revHH{To reduce the cost for forming generalized eigenvalue problems,
an economic version is also considered and tested for 3D examples.}
We show that the condition numbers of the preconditioned systems
are controlled by a given tolerance,
which is used to select coarse basis functions
from the generalized eigenvalue problems.
Numerical results are presented to verify the robustness of the proposed approaches.

\section*{Acknowledgement}
The first author would like to thank to Stefano Zampini for the help
with implementing the deluxe scaling in the change of basis formulation.

\bibliography{adaptive-ref}
\bibliographystyle{plain}

\end{document}